\documentclass[a4paper,11pt]{amsart}
  \usepackage{amsmath}
  \usepackage{url,graphicx}
  \usepackage{amssymb, color, pstricks-add, manfnt}
  \usepackage{amsmath, amsthm,  amsfonts}
  \usepackage{mathrsfs}

  \theoremstyle{plain}
  \newtheorem{Theorem}{Theorem}[section]
  \newtheorem{Lemma}{Lemma}[section]

  \newtheorem{Corollary}{Corollary}[section]

    \theoremstyle{remark}
  \newtheorem{remark}{Remark}

  \numberwithin{equation}{section}
  \numberwithin{figure}{section}
  \numberwithin{remark}{section}

\renewcommand{\baselinestretch}{1.00}
\usepackage{hyperref}
\usepackage{cite}

\begin{document}

\title{Oblique boundary value problems for augmented Hessian equations III}

\author{Feida Jiang}
\address{College of Mathematics and Statistics, Nanjing University of Information Science and Technology, Nanjing 210044, P. R. China}
\email{jfd2001@163.com}

\author{Neil S. Trudinger}
\address{Mathematical Sciences Institute, The Australian National University, Canberra ACT 0200, Australia; School of Mathematics and Applied Statistics, University of Wollongong, Wollongong, NSW 2522, Australia}
\email{Neil.Trudinger@anu.edu.au; neilt@uow.edu.au}

\thanks{Research supported by National Natural Science Foundation of China (No. 11771214),  Australian Research Council (No. DP170100929)}

\subjclass[2010]{35J66, 35J25, 35J70, 35D30}

\date{\today}

\keywords{Oblique boundary value problems, augmented Hessian equations, degenerate equations, local second derivative estimates, existence, uniqueness}

\abstract {In bounded domains, without any geometric conditions, we study the existence and uniqueness of  globally Lipschitz and interior strong $C^{1,1}$, (and classical $C^2$), solutions of general semilinear oblique boundary value problems for degenerate, (and non-degenerate), augmented Hessian equations, with strictly regular associated matrix functions. By establishing local second derivative  estimates at the boundary and proving viscosity comparison principles, we show that the solution is correspondingly smooth near boundary points where the appropriate uniform convexity is satisfied. 
}
\endabstract

\maketitle


\baselineskip=12.8pt
\parskip=3pt
\renewcommand{\baselinestretch}{1.38}

\section{Introduction}\label{Section 1}
\vskip10pt

In this paper, we continue our previous studies \cite{JT-oblique-I, JT-oblique-II} of oblique boundary value problems for elliptic solutions of augmented  Hessian equations and consider in particular problems with general boundaries so that at least our boundary conditions must be interpreted in a weak sense. Following Section 4.3 in \cite{JT-oblique-I}, we also treat the degenerate elliptic case. As in  \cite{JT-oblique-I, JT-oblique-II}, our boundary value problems have the general form 
\begin{equation}\label{1.1}
\mathcal{F}[u]:= F(D^2u-A(\cdot,u,Du))=B(\cdot,u,Du), \quad \mbox{in} \ \Omega,
\end{equation}
\begin{equation}\label{1.2}
\mathcal{G}[u]:=G(\cdot,u,Du)=0, \quad \mbox{on} \ \partial\Omega,
\end{equation}
where $\Omega$ is a bounded domain in $n$ dimensional Euclidean space $\mathbb{R}^n$, $u$ is the scalar unknown function, $Du$ and $D^2u$ denote the gradient vector and the Hessian matrix of $u$, $A$ is a $n\times n$ symmetric matrix function defined on $\Omega \times \mathbb{R} \times \mathbb{R}^n$, $B$ is a  scalar valued function on $\Omega \times \mathbb{R} \times \mathbb{R}^n$ and $G$ is a scalar valued function defined on $\partial \Omega \times \mathbb{R} \times \mathbb{R}^n$.  As usual, we use $x$, $z$, $p$ and $r$ to denote the points in $\Omega$, $\mathbb{R}$, $\mathbb{R}^n$ and $\mathbb{S}^{n}$ respectively, where 
$\mathbb{S}^n$ denotes the linear space of $n\times n$ symmetric matrices. The function $F$ is defined on an open set $\Gamma$ in $\mathbb{S}^{n}$, which is closed under addition of the positive cone. We shall use either $\mathcal{F}$ or $F$ to denote the general operator in \eqref{1.1}, and either $\mathcal{G}$ or $G$ to denote the boundary operator in \eqref{1.2}.
The boundary condition \eqref{1.2} is called oblique if the continuous function $G$ is  strictly increasing with respect to $p$ in the normal direction to $\partial\Omega$ at $x$, namely
\begin{equation}\label{1.3}
\frac{G(x,z,p+\lambda \nu(x))-G(x,z,p)}{\lambda} >0,
\end{equation}
for all $(x,z,p)\in \partial\Omega \times \mathbb{R} \times \mathbb{R}^n$ and $\lambda>0$, where $\nu(x)$ denotes the unit inner normal vector to $\partial\Omega$ at the point $x$. When $G$ is differentiable with respect to $p$, then the obliqueness \eqref{1.3} is equivalent to
\begin{equation}\label{1.4}
D_pG\cdot \nu > 0,
\end{equation}
for all $(x,z,p)\in\partial \Omega \times \mathbb{R} \times \mathbb{R}^n$, where $\nu$ denotes the unit inner normal vector field on $\partial\Omega\in C^1$. In this paper we restrict attention to the semilinear case
\begin{equation}\label{1.5}
\mathcal{G}[u]=G(\cdot, u, Du) = \beta\cdot Du - \varphi(\cdot, u) =0, \quad {\rm on} \ \partial \Omega,
\end{equation}
where $\beta$ is a vector field on $\partial\Omega$, satisfying $\beta \cdot \nu>0$ on $\partial\Omega$, and 
$\varphi$ is a scalar function on $\partial\Omega \times \mathbb{R}$. The standard  example of \eqref{1.5} is the semilinear Neumann boundary condition, where $\beta=\nu$ on $\partial\Omega$.

The oblique derivative problem \eqref{1.1}-\eqref{1.2} for augmented Hessian equations arises naturally  in the theory of fully nonlinear elliptic equations, through its applications in conformal geometry, optimal transportation, and geometric optics; see \cite{Tru2006, LiLi2006, SSChen2007, JLL2007, LiNguyen2014, JT2014, JT-oblique-I, JT-oblique-II, JT2018}. One can refer to \cite{JT-oblique-I} for  more detailed background and various explicit examples of the functions $F$, $G$ and $A$. In \cite{JT-oblique-I, JT-oblique-II}, we have established extensive existence theorems for classical solutions of problem \eqref{1.1}-\eqref{1.2} under appropriate domain convexity hypotheses. In this paper, we prove existence and uniqueness results for solutions in $C^{1,1}(\Omega)\cap C^{0,1}(\bar\Omega)$ in the degenerate  case, (and  $C^2(\Omega)\cap C^{0,1}(\bar\Omega)$ in the non-degenerate case), to the semilinear boundary value problem \eqref{1.1}-\eqref{1.5} on general bounded domains, (without  convexity assumptions), which satisfy  higher regularity near a boundary point where the appropriate uniform convexity is locally satisfied. Our treatment also embraces more general viscosity solutions of equation
\eqref{1.1} and such an extension will be pursued in future work.

Our assumptions for $F$ correspond to  those in \cite{JT-oblique-I}, although here we shall write them for non-smooth $F$.  Assuming now that $F\in C^0(\Gamma)$ where $\Gamma$ is a convex open cone $\subset  \mathbb{S}^n$, $\ne \mathbb{S}^n$, with vertex at $0$, containing the positive cone $K^+$, we formulate the following conditions on $F$, 
\begin{itemize}
\item[{\bf F1}:]
$F$ is strictly increasing in $\Gamma$, that is 
$$\liminf _{\sigma \rightarrow 0} \frac{F(r+\sigma \eta) - F(r)}{\sigma} >0$$
for any positive $\eta \in \mathbb{S}^{n}$;

\vspace{0.2cm}

\item[{\bf F1$^-$}:]
$F$ is nondecreasing in $\Gamma$, (that is $F(r+\eta) \ge F(r)$ for any non-negative $\eta \in \mathbb{S}^{n}$), with 

$$\mathscr{T^-}(r):= \liminf _{\sigma \rightarrow 0} \frac{F(r+\sigma I) - F(r)}{\sigma} > 0;$$

\vspace{0.2cm}

\item[{\bf F2}:]
$F$ is concave in $\Gamma$; 

\vspace{0.2cm}

\item[{\bf F3}:]
$F(\Gamma)=(a_0, \infty)$ for constant $a_0 \ge -\infty $, with
$\sup\limits_{r_0\in \partial\Gamma}\limsup\limits_{r\rightarrow r_0} F(r) \le a_0$;

\vspace{0.2cm}

\item[{\bf F4}:]
$F(tr)\rightarrow \infty$ as $t\rightarrow \infty$, for all $r\in \Gamma$;

\vspace{0.2cm}

\item[{\bf F5}:]
For given constants $a$, $b$ satisfying $a_0<a<b$, there exists a constant $\delta_0>0$ such that
$\mathscr{T^-}(r) \ge \delta_0$
 for all $r$ satisfying  $a< F(r) < b$;
 
 \vspace{0.2cm}

\item[{\bf F5\textsuperscript{+}}:]
$\mathscr{T^-}(r) \rightarrow \infty$ uniformly for $a <F(r) < b$ as $|r|\rightarrow \infty$.

\end{itemize}

These conditions coincide with those in \cite{JT-oblique-I}, when $F \in C^2(\Gamma)$, in which case  $\mathscr{T^-} = \mathscr{T}= {\rm{trace}}( F_r) $. As there, condition F5 is essentially superfluous as it is implied by  F1$^-$, F2, F3 and F4, while F4 is itself implied by F1$^-$, F2 and F3, when $a_0 > - \infty$. However as well as F5\textsuperscript{+}, we also need other refinements of F5, when the constant $\delta_0$ is independent of $a$ or $b$, which we designate respectively as F5(0) and F5($\infty$). In fact our primary examples of one positive homogeneous functions $F \in C^0(\bar \Gamma)$, which are positive, increasing and concave in $\Gamma$ and vanish on $\partial\Gamma$ , satisfy F1, F2, F3, (with $a_0 =0$), F4 and both  F5(0), F5($\infty$). Also in general when $a_0 > -\infty$,  F1$^-$, F2 and F3  imply F5(0), \cite{JT-oblique-I}. Similarly we can refine condition  F5\textsuperscript{+} when $a=a_0$ or $b=\infty$, so that in particular F5\textsuperscript{+}(0) is satisfied by the normalised $k$-Hessians $F_k$ in the cones $\Gamma_k$ for $k = 2,\cdots,n$, \cite{JT-oblique-I}, but F5\textsuperscript{+}($\infty$) is incompatible with F2.


As in \cite{JT-oblique-I, JT-oblique-II}, we call $M[u]: =D^2u-A(\cdot, u, Du)$ the augmented Hessian matrix. Assuming always that $F$ at least satisfies F1$^-$,  a function $u\in C^0(\Omega)$, is then called admissible for $F$ at a point $x_0\in  \Omega$ if $u$ is twice differentiable at $x_0$ and the augmented Hessian matrix $M[u](x_0) \in \bar\Gamma$. It follows then that the operator $\mathcal F$ is degenerate elliptic, with respect to $u$, at $x_0$, that is $F(M[u] + \eta)(x_0) \ge F(M[u])(x_0)$, for all $\eta\ge 0$, $\in \mathbb{S}^{n}$. If $F$ satisfies F1 and $M[u](x_0) \in \Gamma$, then $\mathcal F$ is elliptic, with respect to $u$, at $x_0$. Our weak form of the boundary condition \eqref{1.2} corresponds to that for viscosity solutions \cite {CIL1992}. Namely a function $u \in C^0(\bar\Omega)$ satisfies the inequality $\mathcal{G}[u] \ge 0$, $(\le 0)$, on $\partial\Omega$ weakly, with respect to the operator $\mathcal F$, if for any admissible function $\phi \in C^2(\bar\Omega)$, $x_0\in\partial\Omega$ satisfying $u \le, (\ge), \phi$ in $\bar\Omega$, $u(x_0) = \phi(x_0)$, we have either $\mathcal{G}[\phi](x_0) \ge 0$, $(\le 0)$ or  $\mathcal F[\phi](x_0) \ge, (\le), B(\cdot,\phi,D\phi)(x_0)$. The function $u \in C^0(\bar\Omega)$ then satisfies the boundary condition \eqref{1.2} weakly if  both $\mathcal{G}[u] \ge 0$ and $\mathcal{G}[u]\le 0$ on $\partial\Omega$ weakly. If $F$ satisfies F1, then we only need the inequalities $\mathcal{G}[\phi](x_0) \ge 0$, $(\le 0)$. In this paper we consider solutions $u$ at least in $C^{1,1}(\Omega)$ with \eqref{1.1} satisfied almost everywhere and postpone consideration of viscosity solutions in $C^0(\Omega)$ or $C^{0,1}(\Omega)$, under reduced structure conditions on $F$, to a future work. 

An important ingredient for regularity of solutions to equations involving the augmented matrix $M[u]$ is the co-dimension one convexity (strict convexity) condition on the matrix $A$ with respect to $p$, that is
\begin{equation}\label{strict regularity}
A_{ij}^{kl}(x,z,p)\xi_i\xi_j\eta_k\eta_l \geq 0, \ (>0 ),
\end{equation}
for all  $\xi,\eta\in\mathbb{R}^n$, $\xi\perp\eta$,
where $A_{ij}^{kl}=D^2_{p_kp_l}A_{ij}$ and $A$ is twice differentiable at $(x,z,p)\in\Omega\times \mathbb{R}\times\mathbb{R}^n$. As in \cite{JT-oblique-I}, we will assume at least that the matrix $A\in C^2$ is strictly regular in $\Omega$, that is the strict inequality in \eqref{strict regularity}, holds for all $(x,z,p)\in\Omega\times \mathbb{R}\times\mathbb{R}^n$. Also following  \cite{JT-oblique-I}, we need additional conditions for gradient estimates. In particular, we may strengthen the strict regularity condition by assuming that $A$ is uniformly regular in $\Omega$, namely that for any $M > 0$, there exist positive constants $\lambda_0$ and $\bar\lambda_0$ such that 
\begin{equation}\label{1.7}
\sum_{i,j,k,l}A_{ij}^{kl}\xi_i\xi_j\eta_k\eta_l \geq \lambda_0|\xi|^2|\eta|^2 - \bar\lambda_0(\xi\cdot\eta)^2,
\end{equation}
for all $\xi, \eta \in \mathbb{R}^n$, $x\in\Omega$, $|z|\le M$, $p\in \mathbb{R}^n$. For gradient bounds we also assume the functions $A$ and $B$ to satisfy quadratic growth conditions, analogous to those for quasilinear elliptic equations \cite {GTbook}. Namely,
\begin{equation} \label{quadratic growth}
D_xA, D_xB, D_zA, D_zB = O(|p|^2), \quad D_pA, D_pB = O(|p|),
\end{equation}
as $|p| \rightarrow \infty$, uniformly for $x\in\Omega$, $|z|\le M$ for any $M > 0$. 

For local boundary regularity, we shall assume a local version of our boundary convexity condition in
\cite{JT-oblique-I}. Assuming that $G \in C^0(\bar\Omega \times \mathbb{R} \times \mathbb{R}^n)$, $\partial \Omega\in C^2$, we call $\partial\Omega$ uniformly $(\Gamma, A, G)$-convex at $x_0\in \partial\Omega$, with respect to an interval $\mathcal {I}_0$, if 
\begin{equation}\label{curvature condition}
K_A(\partial\Omega) (x_0, z, p) + \mu \nu(x_0) \otimes \nu(x_0) \in \Gamma,
\end{equation}
for all  $z\in \mathcal{I}_0$, $G(x_0,z,p)\ge 0$, and some $\mu=\mu(x_0,z,p)>0$, where $K_A $ denotes the $A$-curvature matrix of $\partial\Omega$, given by
\begin{equation}
K_A( \partial\Omega) (x, z, p) = - \delta \nu(x) + P(D_pA(x,z,p)\cdot \nu(x))P
\end{equation}
at any point $x\in\partial\Omega$, where $\partial\Omega$ is twice differentiable, $\nu$ denotes the unit inner normal, $\delta=D-(\nu\cdot D)\nu$ denotes the tangential gradient and $P=I-\nu\otimes\nu$ is the projection matrix onto the tangent space. Also corresponding to our global definitions in \cite{JT-oblique-I} we call $\partial\Omega$ uniformly $(\Gamma, A, G)$-convex at $x_0\in \partial\Omega$, with respect to
 $u\in C^0 (\partial\Omega)$, if  $\mathcal {I}_0 = \{u(x_0)\}$. We also recall our previous terminology in \cite{JT-oblique-I} that $\mathcal P_k$, for $k= 1,\cdots, n$, denotes the cone in $\mathbb{S}^{n}$ where the sum of any $k$ eigenvalues is positive.
 
 The scope of our results is embodied in the following theorem, which covers both the degenerate and non-degenerate cases. As in \cite{JT-oblique-I}, we also assume the existence of sub and supersolutions  in order to have {\it a priori} solution bounds for our approximating problems. More general results and alternative hypotheses will be treated in conjunction with our proofs.

\begin{Theorem}\label{Th1.1}
Assume that $F$ satisfies conditions F1$^-$, F2 and F3, with $a_0 = 0$,  in $\Gamma$, $\Omega$ is a $C^{2,1}$ bounded domain in $\mathbb{R}^n$, $A\in C^2(\bar \Omega\times \mathbb{R}\times \mathbb{R}^n)$ is uniformly regular in $\Omega$, $B \ge 0, \in C^2(\bar \Omega\times \mathbb{R}\times \mathbb{R}^n)$, 
$\mathcal G$ is semilinear and oblique with $\beta\in C^{1,1}(\partial\Omega)$, $\varphi \in C^{1,1}(\partial\Omega\times \mathbb{R})$ and there exist a strict subsolution $\underline u\in  C^{1,1}(\bar \Omega)$ of \eqref{1.1}-\eqref{1.5} and a supersolution $\bar u \in  C^{1,1}(\bar \Omega)$ of \eqref{1.1}. Assume also that $A$ and $B$ satisfy the quadratic growth conditions \eqref{quadratic growth}, $A$, $B$ and $\varphi$ are nondecreasing with respect to $z$, with one of them strictly increasing and either (a) B is independent of $p$  or (b) $B$ is convex in $p$ and $F$ satisfies  F5\textsuperscript{+}(0) and F5($\infty$). 
Then we have the following:

{\rm (i)} there exists an admissible solution $u\in C^{1,1}(\Omega)\cap C^{0,1}(\bar\Omega)$ of the boundary value problem \eqref{1.1}-\eqref{1.5}; 

{\rm (ii)} the solution $u$ is unique if either $A$ is strictly increasing in $z$ or $\varphi$ is strictly increasing in $z$, with $A$ and $B$ independent of $z$;

{\rm (iii)} if also $\Gamma \subset \mathcal P_{n-1}$ and $\partial\Omega$ is uniformly $(\Gamma, A, G)$-convex at $x_0\in \partial\Omega$, with respect to $u$, then $u\in C^{1,1}(\Omega\cup (\mathcal N\cap \bar\Omega))$ for some neighbourhood $\mathcal N$ of $x_0$; 

{\rm (iv)} if $F$ also satisfies F1 and $B > 0$, then $u \in C^{2,\alpha}(\Omega\cup(\mathcal N\cap\bar\Omega))$ for some $\alpha>0$.
\end{Theorem}

\begin{remark}\label{Rem 1.1}
Under our hypothesis of uniform regularity of $A$, Theorem \ref{Th1.1}  is established for general oblique boundary value problems \eqref{1.5}. When $A$ is assumed only strictly regular, the conclusions in Theorem \ref{Th1.1} still hold under further conditions on $F$, $\beta$, $A$ and $B$. To state these,  we write  condition F7 from \cite{JT-oblique-I} in the weak form:
\begin{itemize}
\item[{\bf F7}:]
For a given constant $a>a_0$, there exists constants $\delta_0, \delta_1>0$ such that 
$$\liminf _{\sigma \rightarrow 0} \frac{F(r+\sigma \xi\otimes \xi) - F(r)}{\sigma} \ge \delta_0 + \delta_1 \mathscr{T}^+(r),$$ 
if $a\le F(r)$, $\xi$ is a unit eigenvector of $r$ corresponding to a negative eigenvalue and 

$$\mathscr{T^+}(r):= \limsup _{\sigma \rightarrow 0} \frac{F(r+\sigma I) - F(r)}{\sigma}.$$

\end{itemize}
Consistent with our notation above, when the constants $\delta_0, \delta_1>0$ are independent of $a$, we will refer to condition F7 as F7(0). Then for strictly regular $A$, the conclusions in Theorem \ref{Th1.1} hold if
$|\frac{\beta}{\beta\cdot\nu} - \nu|<1/\sqrt{n}$, $\mathcal{F}$ is orthogonally invariant and either 
(a) $F$ satisfies condition F7(0), (or F7 if $\inf B> 0$), and $A=o(|p|^2)I$ in \eqref{quadratic growth} or (b) $A=o(|p|^2)I$, $p\cdot D_pA \le o(|p|^2)I$, $p\cdot D_pB\le o(|p|^2)$ in \eqref{quadratic growth} or (c) $\beta=\nu$, $\Gamma\subset \Gamma_k$ with $k>n/2$ and $\Omega$ is convex. These alternative conditions to the uniform regularity of $A$ are used for the gradient estimates; see Theorems 1.3 and 3.1, and Remark 3.3 in \cite{JT-oblique-I} for more details, as well as Section  \ref{Section 3}, where we will show that the hypothesis F7 in Theorem 1.3 of \cite{JT-oblique-I} can be replaced by F2 and F5($\infty$), under the stronger growth conditions in (b).

As introduced in \cite{JT-oblique-I}, typical examples of uniformly regular matrices $A$ are given by
\begin{equation}\label{typical examples for A}
A(x,z,p)=\frac{1}{2}a_{kl}(x,z)p_kp_l I - a_0(x,z)p\otimes p,
\end{equation}
where $a_{kl}, a_0\in C^2(\bar \Omega\times \mathbb{R})$ and the matrix $\{a_{kl}\}>0$ in $\bar\Omega \times \mathbb{R}$. When $a_{kl}=\delta_{kl}$ and $a_0=1$, \eqref{typical examples for A} is related to the Schouten tensor in conformal deformation, see \cite{Tru2006, JT-oblique-I}. Note also that in Theorem \ref{Th1.1}, the operator $F$ is in the class $C^0(\Gamma)$ and need not be orthogonally invariant so that in particular Theorem \ref{Th1.1} readily applies to the Bellman type augmented Hessian equations in Section 4.2 in \cite{JT-oblique-I}.
\end{remark}

\begin{remark}\label{Rem 1.2}
We also have the uniqueness in assertion (ii) of Theorem \ref{Th1.1}  if $B$ is strictly increasing in $z$ and $\mathscr{T}(r)$ is bounded from above; (see case (iii) in Theorems 4.2 and 4.3). In general, without uniqueness, if the monotonicity, subsolution and supersolution hypotheses in Theorem \ref{Th1.1} are replaced by conditions (4.6) and (4.7) in \cite{JT-oblique-I}, the existence in assertion (i) of Theorem \ref{Th1.1} still holds by using the Leray-Schauder principle, Theorem 5.1 in \cite{GDbook}. 
\end{remark}

\begin{remark}\label{Rem 1.3}
 Clearly if $\partial\Omega$  is uniformly $(\Gamma, A, G)$-convex with respect to $u$ at every boundary point, we have $u\in C^{1,1}(\bar \Omega)$ and $u\in C^{2,\alpha}(\bar \Omega)$ in assertions (iii) and (iv) of Theorem \ref{Th1.1} respectively and we recover in particular our classical existence results in \cite{JT-oblique-I}.
\end{remark}
 
 \begin{remark}\label{Rem 1.4}
 It will also be clear from our proofs that the condition $F(\Gamma) = (a_0,\infty)$ in F3 can be replaced by $F(\Gamma) \supset (a_0,\infty)$, (assuming also F4 for (iii) and (iv)). Moreover, we can alternatively replace in Theorem \ref{Th1.1}, in accord with \cite{ITW2004}, the cone $\Gamma$ by any convex open set $D\subset \mathbb{S}^n$, $\ne \mathbb{S}^n$, with $0\in\partial D$, which is closed under addition of the positive cone, $K^+$, provided its asymptotic cone $\Gamma$ is used in our hypotheses of ($\Gamma, A, G$) -convexity and we assume $F$ also satisfies F5(0) in $D$, (or F5 if $\inf B > a_0$). Here the asymptotic cone of $D$ is defined as the largest convex open cone $\Gamma$, with vertex at $0$, such that $D$ is closed under addition of $\Gamma$. In this more general situation conditions  F1$^-$, F2 and F3 continue to imply F4 and F5(0) hold in the asymptotic cone $\Gamma$ but not necessarily in $D$. 
 \end{remark}

The organization of this paper is as follows.

In Section \ref{Section 2}, we consider local boundary second derivative estimates for solutions of \eqref{1.1}-\eqref{1.5} in the mixed tangential-oblique, pure tangential, and pure oblique directions, which are the local versions of the corresponding global boundary estimates in Section 2 in \cite{JT-oblique-I}. The resultant local boundary second derivative estimates are established in Theorem \ref{Th2.1}, through modification of the proof of Theorem 1.2 in \cite{JT-oblique-I}, in accordance with Remarks 2.1 and 3.6 in \cite{JT-oblique-I}. Note that in this section we always assume $F\in C^2(\Gamma)$, which suffices for obtaining the local estimates for the approximating problems with smooth operators.

In Section \ref{Section 3}, we prove in Theorem \ref{Th3.1} the existence of solutions to the problem \eqref{1.1}-\eqref{1.5} in $C^{1,1}(\Omega)\cap C^{0,1}(\bar\Omega)$ ($C^{2,\alpha}(\Omega)\cap C^{0,1}(\bar\Omega)$ for some $\alpha\in (0,1)$) under F1$^-$ (F1) and $B\ge 0$ ($B>0$) by solving the regularized problem and using approximations. We also consider alternative hypotheses for the gradient estimates, in accord with our Remark \ref{Rem 1.1} above, leading to more general versions of Theorems \ref{Th1.1} and \ref{Th3.1}.

In Section \ref{Section 4}, we first prove a comparison principle in Theorem \ref{Th4.1} for solutions of the equation \eqref{1.1} based on a barrier construction. Under proper monotonicity assumptions on $A$, $B$ and $G$, we then study the comparison principles for solutions of \eqref{1.1}-\eqref{1.5} in $C^{2}(\Omega)\cap C^{0,1}(\bar \Omega)$ and $C^{1,1}(\Omega)\cap C^{0,1}(\bar \Omega)$ in Theorems \ref{Th4.2} and \ref{Th4.3} respectively, which gives the uniqueness for solutions obtained in Section \ref{Section 3} in corresponding function spaces. Moreover, if there exists a boundary point satisfying the curvature condition \eqref{curvature condition}, we show higher regularity of $u$ in a boundary neighbourhood with the help of the local second derivative estimate in Theorem \ref{Th2.1}. Finally, the assertions of Theorem \ref{Th1.1} are proved.

To conclude the introduction, we remark that the  results here are already  foreshadowed in Section 4.4 in \cite{JT-oblique-I}.
Namely in the nondegenerate (or degenerate) case, if we drop the domain convexity conditions, we still infer existence of classical (or strong) solutions of equation \eqref{1.1} which are globally Lipschitz continuous and satisfy the boundary condition \eqref{1.5} in a weak viscosity sense so that the domain convexity conditions become conditions for boundary regularity.

The notation of this paper, unless otherwise specified, follows \cite{JT-oblique-I, JT-oblique-II}.

\vspace{3mm}


\section{Local second derivative estimates}\label{Section 2}
\vskip10pt

In this section, we prove local second derivative estimates for the solutions of oblique boundary value problem \eqref{1.1}-\eqref{1.5}, as already asserted in Remarks 2.1 and 3.6 in \cite{JT-oblique-I}. Since we already have the interior second derivative estimate (1.14) in Theorem 1.1 in \cite{JT-oblique-I}, here we only need to focus on the local second derivative estimates near the boundary.

For a fixed point $x_0\in \bar\Omega$ and a positive constant $R$, we use $B_R:=B_R(x_0)$ to denote the ball of radius $R$ and centre  $x_0$.  For $x_0\in \partial\Omega$, we assume in this section that $\nu$ and $\varphi$ in $B_R\cap \partial\Omega$ have been smoothly extended to $\bar B_R\cap \bar \Omega$, so that $\nu$ and $\varphi(\cdot, z)$ are constant along normals of $B_R(x_0)\cap\partial\Omega$.
First, by differentiating the boundary condition \eqref{1.5} with respect to a tangential vector field $\tau$, we obtain
\begin{equation}\label{tangential diff G}
D_{\beta\tau}u = \tilde D_{x_\tau} \varphi - D_\tau \beta \cdot Du, \quad {\rm on} \ B_R\cap \partial\Omega,
\end{equation}
where $\tilde D_{x_\tau}=\tau\cdot \tilde D_x$ and $\tilde D_x = D_x + Du D_z$.
Then from \eqref{tangential diff G} we have an estimate for mixed tangential-oblique second order derivatives,
\begin{equation}\label{tan obl}
|u_{\tau\beta}| \le C, \quad {\rm on} \  B_R\cap\partial\Omega,
\end{equation}
for any unit tangential vector field $\tau$, where the constant $C$ depends on $B_R\cap\Omega$, $\beta$, $\varphi$ and $|u|_{1;B_R\cap\Omega}$.


We next treat the local pure tangential second order derivative estimates for $u$, for which the strictly regular condition \eqref{strict regularity} for the matrix $A$ is critical. As in \cite{JT-oblique-I}, it is convenient here to use \eqref{strict regularity} to express the strict regularity of $A$ with respect to $u$, in the form
\begin{equation}\label{strict regularity c1}
A_{ij}^{kl}(\cdot, u, Du) \xi_i\xi_j\eta_k\eta_l \ge c_0 |\xi|^2|\eta|^2 - c_1 (\xi\cdot \eta)^2,
\end{equation}
for arbitrary vectors $\xi, \eta\in \mathbb{R}^n$, where $c_0$ and $c_1$ are positive constants depending on $A$ and $\sup(|u|+|Du|)$. 
Since the estimates will depend on the obliqueness, we can assume that
\begin{equation}\label{ob}
\beta \cdot \nu \ge \beta_0, \quad {\rm on} \ \partial\Omega,
\end{equation}
where $\beta_0$ is a positive constant.
Setting
\begin{equation}
M_2(R)=\sup_{B_R\cap \Omega}|D^2u|, 
\end{equation}
we formulate the local pure tangential second order derivative estimates on the boundary in terms of $M_2(R)$ as follows.
\begin{Lemma}\label{Lemma 2.1}
Let $u\in C^2(\bar \Omega)\cap C^4(\Omega)$ be an admissible solution of the boundary value problem \eqref{1.1}-\eqref{1.2} in a bounded domain $\Omega\subset \mathbb{R}^n$ with $\partial\Omega\in C^{2,1}$. Assume that $F\in C^2(\Gamma)$ satisfies conditions F1-F3 and F5 in the cone $\Gamma\subset \mathbb{S}^n$, $A\in C^2(\bar \Omega \times \mathbb{R} \times \mathbb{R}^n)$ is strictly regular, $B>a_0, \in C^2(\bar\Omega\times \mathbb{R} \times \mathbb{R}^n)$, $G\in C^2(\partial\Omega \times \mathbb{R} \times \mathbb{R}^n)$ is semilinear and oblique satisfying  \eqref{1.5} and \eqref{ob}, and either F5\textsuperscript{+} holds or $B$ is convex with respect to $p$. Then for any $x_0\in \partial\Omega$, $0<R<1$ and ball $B_R=B_R(x_0)$, we have the estimate
\begin{equation}\label{local pure + estimate}
\sup_{|\tau|=1, \tau\cdot \nu=0} u_{\tau\tau}(x_0) \le \epsilon M_2(R) +  C_\epsilon(1+\frac{1}{R^2}),
\end{equation}
for any positive constant $\epsilon>0$, where $\nu$ is the unit inner normal vector at $x_0$, $C_\epsilon$ is a constant depending on $\epsilon, B_R\cap \Omega, F, A, B, \varphi, \beta, \beta_0$ and $|u|_{1;B_R\cap \Omega}$. Furthermore, if $\Gamma\subset \mathcal{P}_{n-1}$, then for any constant $\epsilon>0$, we have the estimate
\begin{equation}\label{local pure prime estimate}
\sup_{|\tau|=1, \tau\cdot \nu=0} |u_{\tau\tau}(x_0)| \le \epsilon M_2(R) +  C_\epsilon(1+\frac{1}{R^2}),
\end{equation}
for any positive constant $\epsilon>0$, where $C_\epsilon$ is again a constant depending on $\epsilon, B_R\cap \Omega, F, A, B, \varphi, \beta, \beta_0$ and $|u|_{1;B_R\cap \Omega}$.
\end{Lemma}


The local pure tangential estimates \eqref{local pure + estimate} and \eqref{local pure prime estimate} for semilinear $G$ depend on the cut-off function $\zeta$ constructed in Section 3.2 of \cite{JT-oblique-I}. For $x_0\in \partial\Omega$ and a sufficiently small positive constant $R$, there exists a cut-off function $\zeta\in C^0(\bar B_R(x_0))\cap C^2(S_\zeta)$, such that 
\begin{equation}\label{cutoff 1}
0\le \zeta \le 1 \ {\rm in}\  \bar B_R(x_0)\cap S_\zeta,  \quad \zeta(x_0)=1,
\end{equation}
\begin{equation}\label{cutoff 2}
\zeta =0 \ {\rm on} \  \partial B_R(x_0)\cap S_\zeta, \quad D_\beta \zeta=0 \ {\rm on} \ \partial\Omega \cap S_\zeta,
\end{equation}
and
\begin{equation}\label{cutoff 3}
|D\zeta| \le C/R, \  |D^2\zeta| \le C/R^2,  \ {\rm in} \  B_R(x_0)\cap S_\zeta,
\end{equation}
where $S_\zeta$ is the support of $\zeta$, $\beta$ is a vector field on $\partial\Omega$ satisfying \eqref{ob} and $C$ is a positive constant; (see the construction of such $\zeta$ at the end of the proof for Theorem 3.1 in \cite{JT-oblique-I}). 


\begin{proof}[Proof of Lemma \ref{Lemma 2.1}]
For $x_0\in \partial\Omega$, we first fix a sufficiently small positive $R$ with $\varphi, \beta$ and $\nu$ in $B_R\cap \partial\Omega$ smoothly extended to $\bar B_R\cap \bar\Omega$ and a cut-off function $\zeta$ satisfying \eqref{cutoff 1}, \eqref{cutoff 2} and \eqref{cutoff 3}, where $B_R=B_R(x_0)$. 
Setting 
\begin{equation}
w_{\tau} = u_{\tau\tau} + \frac{c_1}{2} |u_\tau|^2,
\end{equation}
we suppose that the function
\begin{equation}\label{2.50}
v_\tau =\zeta^2 w_{\tau}
\end{equation}
takes a maximum over $\bar B_R\cap\partial\Omega$ and unit tangential vectors $\tau$, at a point $y_0 \in\bar B_R\cap \partial\Omega$ and vector $\tau=\tau_0$ where $c_1$ is the constant in the strict regularity condition \eqref{strict regularity c1}. Without loss of generality, we may assume $y_0=0$ and $\tau_0=e_1$. Setting $b=\frac{\nu_1}{\beta\cdot \nu}$ and $\tau=e_1 - b\beta$, we then have, at any point in $B_R\cap\partial\Omega$,
\begin{equation}
v_1 = v_\tau + \zeta^2 [b(2u_{\beta\tau}+c_1u_\beta u_\tau )+ b^2 (u_{\beta\beta} + \frac{c_1}{2}|u_\beta|^2)],
\end{equation}
with $v_1=v_{e_1}$, $v_1(0)=v_\tau (0)$, $b(0)=0$ and $\tau(0)=e_1$. Setting
$$g: =\frac{1}{\beta\cdot \nu} (2u_{\beta\tau}+c_1u_\beta u_\tau ),$$
from \eqref{tangential diff G}, we have
\begin{equation}
|g-g(0)|\le  C(1+M_2(R))|x|, \quad {\rm on}\ B_R\cap \partial\Omega,
\end{equation}
where $C$ is a constant depending on $B_R\cap\partial\Omega, \varphi, \beta, \beta_0$ and $|u|_{1;B_R\cap\Omega}$. 
Accordingly, we have
\begin{equation}
v_1 - g(0)\zeta^2\nu_1 \le v_\tau + C_1(1+M_2(R)) \zeta^2 |x|^2 , \quad {\rm on} \ B_R\cap\partial\Omega,
\end{equation}
for a further constant $C_1$ depending on the same quantities.
Therefore, the function
\begin{equation}\label{2.51}
\begin{array}{rl}
\tilde v_1 : = \!\!&\!\!\displaystyle v_1 - \zeta^2[ g(0)\nu_1 + C_1(1 + M_2(R))|x|^2] \\
                 = \!\!&\!\!\displaystyle \zeta^2 [w_1 - g(0)\nu_1 - C_1(1 + M_2(R))|x|^2]
\end{array}
\end{equation}
satisfies
\begin{equation}\label{2.52}
\begin{array}{ll}
\tilde v_1 \!\!&\!\!\displaystyle \le v_\tau = \zeta^2|\tau|^2 w_{\tau/|\tau|} \le \zeta^2(0)w_1(0)|\tau|^2 \\
           \!\!&\!\!\displaystyle \le \zeta^2(0)w_1(0)(1-2b\beta_1 + b^2\sup |\beta|^2) \\
           \!\!&\!\!\displaystyle \le \zeta^2(0)w_1(0)(1-2\frac{\beta_1}{\beta\cdot \nu}(0)\nu_1 + C_1|x|^2)  \\
           \!\!&\!\!\displaystyle := \zeta^2(0)w_1(0) f , \quad {\rm on}\ B_R\cap \partial\Omega,
\end{array}
\end{equation}
where $w_1(0)=w_{e_1}(0)$, $f$ is a function in $C^2(\bar B_R\cap \bar \Omega)$ with $f(0)=1$. Here note that we can assume that $w_1(0)>0$, otherwise we have already obtained the pure tangential estimate from \eqref{2.50}. Moreover, by choosing $C_1$ in \eqref{2.52} sufficiently large, the function $f$ can be made positive in $\bar B_{R}\cap \bar \Omega$.
Now differentiating the boundary condition $D_\beta u- \varphi(\cdot, u)=0$ twice in a tangential direction $\tau$ with $\tau(0)=e_1$, and using \eqref{tan obl},  we obtain
\begin{equation}\label{2.56}
D_\beta u_{11}(0) \ge - C_1 (1 +  M_2(R)).
\end{equation}
By denoting 
\begin{equation}\label{tilde w1}
\tilde w_1:= w_1 - g(0)\nu_1 - C_1(1+M_2(R))|x|^2, 
\end{equation}
we have
\begin{equation}\label{Dbeta tilde v1 f}
\begin{array}{rl}
D_\beta [\tilde v_1 - \zeta^2(0)w_1(0) f ]
= \!\!&\!\!\displaystyle D_\beta (\zeta^2 \tilde w_1) - \zeta^2(0)w_1(0) D_\beta f \\
= \!\!&\!\!\displaystyle \zeta^2 D_\beta \tilde w_1- \zeta^2(0)w_1(0) D_\beta f, \quad {\rm on}\ B_R\cap \partial\Omega,
\end{array}
\end{equation}
where $D_\beta \zeta=0$ on $B_R\cap \partial\Omega$ is used.
From \eqref{2.56} and \eqref{Dbeta tilde v1 f}, at the point $0$ we have,
\begin{equation}\label{Dbeta v1}
D_\beta [\tilde v_1 - \zeta^2(0)w_1(0) f](0) \ge - C_1 \zeta^2(0)(1+M_2(R)),
\end{equation}
for a further constant $C_1$ depending on the same quantities. We then employ a new function
\begin{equation}\label{2.57}
v:= \tilde v_1 - \zeta^2(0)w_1(0) f  - K(1 + M_2(R))\zeta^2\phi,
\end{equation}
where $K$ is a constant to be determined, $\phi \in C^2(\bar \Omega)$ is a negative defining function for $\Omega$ satisfying $\phi=0$ on $\partial\Omega$, $D_\nu \phi=-1$ on $\partial\Omega$.
By fixing a large constant $K$ such that $K>C_1/\beta_0$, we have 
\begin{equation}\label{reason 1 for inter max}
\begin{array}{rl}
D_\beta v(0) = \!\!&\!\!\displaystyle D_\beta [\tilde v_1 - \zeta^2(0) w_1(0) f](0) - D_\beta [K(1 + M_2(R))\zeta^2\phi](0) \\
                  \ge \!\!&\!\!\displaystyle - C_1 \zeta^2(0)(1+M_2(R)) - K \zeta^2(0) (1 + M_2(R)) D_\beta\phi(0) \\
                  \ge \!\!&\!\!\displaystyle (K\beta_0 - C_1)\zeta^2(0)(1+M_2(R)) \\
                     > \!\!&\!\!\displaystyle 0,
\end{array}
\end{equation}
where \eqref{Dbeta tilde v1 f} and $D_\beta\zeta(0)=0$ are used in the first inequality, $D_\nu\phi(0)=-1$ and the obliqueness $\beta\cdot \nu\ge \beta_0$ are used in the second inequality. From \eqref{2.52} and since $\phi=0$ on $\partial\Omega$, it is readily checked that 
\begin{equation}\label{reason 2 for inter max}
v \le 0 \ {\rm on} \  B_R\cap \partial\Omega, \quad {\rm and} \ v(0)=0.
\end{equation}
Since $\zeta=0$ on the inner boundary $\partial B_R\cap \Omega$, $w_1(0)>0$ and $f>0$ on $\partial B_R\cap \Omega$, we have
\begin{equation}\label{reason 3 for inter max}
v = - \zeta^2(0) w_1(0) f \le 0, \quad {\rm on} \ \partial B_R \cap \Omega.
\end{equation}
Then from \eqref{reason 1 for inter max}, \eqref{reason 2 for inter max} and \eqref{reason 3 for inter max} it follows that the function $v$ defined in \eqref{2.57} must take its maximum over $\bar B_R\cap \bar \Omega$ at an interior point $\tilde y_0 \in B_R\cap \Omega$. This effectively reduces our argument to the proof of Theorem 1.1 in \cite{JT-oblique-I}. 

For completeness, we present the remaining proof. As in \cite{JT-oblique-I}, we define the linearized operator
\begin{equation}\label{linearized op full}
\mathcal{L}:= F^{ij}[D_{ij}-A_{ij}^k (\cdot, u, Du)D_k] - (D_{p_k}B(\cdot, u, Du))D_k,
\end{equation}
where $F^{ij}=\frac{\partial F}{\partial r_{ij}}(M[u])$, $A_{ij}^k =D_{p_k}A_{ij} $. By  differentiating equation \eqref{1.1} and using condition F2 and strict regularity \eqref{strict regularity c1}, we have by calculations,
\begin{equation}\label{mathcal Lv 1}
\mathcal{L}v \ge c_0 \zeta^2 \mathscr{T} |Du_1|^2 - C(1+\frac{1}{R^2})(1+\mathscr{T})(1+M_2(R)) + \lambda_B \zeta^2 |Du_1|^2,
\end{equation}
at the maximum point $\tilde y_0$, where the property \eqref{cutoff 3} of the cut-off function $\zeta$ is used, $\mathscr{T}={\rm trace}(F^{ij})$, $c_0$ is the constant in \eqref{strict regularity c1}, $C$ is a constant depending on $B_R\cap\partial\Omega, A, B, \varphi, \beta, \beta_0$ and $|u|_{1;B_R\cap\Omega}$ and $\lambda_B$ is the minimum eigenvalue of the matrix $D^2_pB$. Invoking conditions F5\textsuperscript{+}, or F5 and convexity of $B$ in $p$, from \eqref{mathcal Lv 1} we have
\begin{equation}\label{mathcal Lv 2}
\mathcal{L}v \ge \frac{c_0}{2} \zeta^4 \mathscr{T} |u_{11}|^2 - C(1+\frac{1}{R^2})\mathscr{T}(1+M_2(R)),
\end{equation}
at $\tilde y_0$, for a further constant $C$ depending in addition on $F$. Since $\mathcal{L}v(\tilde y_0) \le 0$, we derive from \eqref{mathcal Lv 2} that
\begin{equation} \label {interior est}
\zeta^2(\tilde y_0)u_{11}(\tilde y_0) \le \sqrt{\frac{2C}{c_0}(1+\frac{1}{R^2})(1 + M_2(R))} \le \epsilon M_2(R)+C_\epsilon(1+\frac{1}{R^2}),
\end{equation}
where $\epsilon$ is any positive constant, $C_\epsilon$ is a constant depending on $\epsilon, B_R\cap\partial\Omega, F, A, B, \varphi, \beta, \beta_0$ and $|u|_{1;B_R\cap\Omega}$ and Cauchy's inequality is used in the second inequality. 
Since $v$ takes its maximum in $\bar B_R\cap \bar \Omega$ at $\tilde y_0$, we have
\begin{equation}\label{reduce tilde y0 to y0}
\begin{array}{rl}
      v(0) \le \!\!&\!\!\displaystyle v(\tilde y_0) = \tilde v_1 (\tilde y_0) - w_1(0)\zeta^2(\tilde y_0)f(\tilde y_0)  - K(1+M_2(R)) \zeta^2(\tilde y_0) \phi(\tilde y_0) \\
             \le \!\!&\!\!\displaystyle \zeta^2 (\tilde y_0)[(u_{11}(\tilde y_0) + \frac{c_1}{2}|u_1(\tilde y_0)|^2)- g(0)\nu_1(\tilde y_0) - C_1(1 + M_2(R))|\tilde y_0|^2  \\
                  \!\!&\!\!\displaystyle  - K(1+M_2(R)) \phi(\tilde y_0) ] - \zeta^2(0)w_1(0)f(\tilde y_0).
\end{array}
\end{equation}
By fixing $\phi$ in \eqref{2.57} and the constant $C_1$ in \eqref{2.52} so that $\phi\ge -\frac{\epsilon}{4K}$ and $f\ge\frac{1}{2}$ in $B_R\cap\Omega$, and using $v(0)=0$ and \eqref{interior est}, we derive from \eqref{reduce tilde y0 to y0} that
\begin{equation}\label{v10 est}
v_1(0) = \zeta^2(0)w_1(0) \le \epsilon M_2(R)+C_\epsilon(1+\frac{1}{R^2}),
\end{equation}
for any constant $\epsilon>0$, and constant $C_\epsilon$ depending on $\epsilon, F, B_R\cap\partial\Omega, A, B, \varphi, \beta, \beta_0$ and $|u|_{1;B_R\cap\Omega}$. Since $v_\tau$ in \eqref{2.50} takes its maximum over $\bar B_R\cap\partial\Omega$ and unit tangential vectors $\tau$, at the point $0$ and the vector $e_1$, we have
\begin{equation}\label{vtau est}
v_\tau \le v_1(0), \quad {\rm on} \ B_R\cap \partial\Omega,
\end{equation}
for any unit tangential vector $\tau$. From \eqref{v10 est} and \eqref{vtau est}, since $\zeta(x_0)=1$, we have
\begin{equation}\label{utautau est}
u_{\tau\tau}(x_0) \le \epsilon M_2(R)+C_\epsilon(1+\frac{1}{R^2}),
\end{equation}
for any constant $\epsilon>0$, and any unit tangential vector $\tau$, where $C_\epsilon$ is a constant depending on $\epsilon, B_R\cap\partial\Omega, F, A, B, \varphi, \beta, \beta_0$ and $|u|_{1;B_R\cap\Omega}$. 
Then the estimate \eqref{local pure + estimate} follows by taking the supremum of \eqref{utautau est} over the unit tangential vectors at the point $x_0$. Furthermore, if $\Gamma\subset \mathcal{P}_{n-1}$, then the sum of any $n-1$ eigenvalues of the augmented Hessian matrix $M[u]$ is positive. Consequently, the estimate \eqref{local pure prime estimate} directly follows from \eqref{local pure + estimate}.

We have proved this lemma in the case when $R$ is sufficiently small. When $R$ is larger, we can first repeat the above argument in $B_{R^\prime} \cap \Omega$ for a fixed sufficiently small $R^\prime$, so that \eqref{local pure + estimate} and \eqref{local pure prime estimate} hold with $R$ replaced by $R^\prime$. Namely, we have
\begin{equation}\label{small sigma case}
\sup_{|\tau|=1, \tau\cdot \nu=0} u_{\tau\tau}(x_0)  \ ({\rm or} \ \sup_{|\tau|=1, \tau\cdot \nu=0} |u_{\tau\tau}(x_0)| ) \le \epsilon M_2(R^\prime) + C_\epsilon(1+\frac{1}{(R^\prime)^2}).
\end{equation}
For the fixed constant $R^\prime$, if we still denote $\frac{C_\epsilon}{(R^\prime)^2}$ by $C_\epsilon$ in \eqref{small sigma case}, then we get the estimates \eqref{local pure + estimate} and \eqref{local pure prime estimate} from \eqref{small sigma case}, since $0<R<1$ and $M_2(R^\prime) \le M_2(R)$ for $R^\prime<R$.
\end{proof}

In order to obtain the local second order derivative estimate on the boundary in pure oblique directions, we further assume that $F$ satisfies condition F4  and $\partial\Omega$ is uniformly $(\Gamma, A, G)$-convex at a boundary point, with respect to $u$. However, the strict regularity condition \eqref{strict regularity} for the matrix $A$ is not needed.
Settting
\begin{equation}
M_2^\prime(R)= \sup\limits_{B_{R}\cap \partial\Omega}\sup\limits_{|\tau|=1,\tau\cdot \nu=0} |u_{\tau\tau}|,
\end{equation}
we formulate the local pure oblique second order derivative estimates on the boundary in terms of $M_2(R)$ and $M_2^\prime(R)$ as follows.
\begin{Lemma}\label{Lemma 2.2}
Let $u\in C^2(\bar \Omega)\cap C^4(\Omega)$ be an admissible solution of the boundary value problem \eqref{1.1}-\eqref{1.2} in a bounded domain $\Omega\subset \mathbb{R}^n$ with $\partial\Omega\in C^{2,1}$. Assume that $F\in C^2(\Gamma)$ satisfies conditions F1-F5 in the cone $\Gamma\subset \mathbb{S}^n$, $A\in C^1(\bar \Omega \times \mathbb{R} \times \mathbb{R}^n)$, $B>a_0, \in C^1(\bar\Omega\times \mathbb{R} \times \mathbb{R}^n)$, $G\in C^2(\partial\Omega \times \mathbb{R} \times \mathbb{R}^n)$ is semilinear and oblique satisfying  \eqref{1.5} and \eqref{ob}, either F5\textsuperscript{+} holds or $B$ is independent of $p$. Assume also $\partial\Omega$ is uniformly $(\Gamma, A, G)$-convex at $x_0\in \partial\Omega$, with respect to $u$.
Then we have
\begin{equation}\label{local pure obl est}
u_{\beta\beta} (x_0) \le \epsilon M_2(R) + C_\epsilon (1+ M_2^\prime(R)) + \frac{C}{R^2},
\end{equation}
for any $0<R<1$ and any $\epsilon>0$, where $C$ is constant depending on $B_R\cap \Omega, A, \varphi, \beta$ and $|u|_{1;B_R\cap\Omega}$, and $C_\epsilon$ is a constant depending on $\epsilon, B_R\cap \Omega, F, A, B, \varphi, \beta, \beta_0$ and $|u|_{1;B_R\cap \Omega}$.
\end{Lemma}

Lemma \ref{Lemma 2.2} is a local version of the pure oblique second order derivative estimate, Lemma 2.2 in \cite{JT-oblique-I}; (see Remark 2.1 there).  In the proof of Lemma 2.2 in \cite{JT-oblique-I}, we employ the barrier argument in a boundary strip $\Omega_\rho:=\{x\in \Omega |\  {\rm dist}(x,\partial\Omega) <\rho\}$, since the uniform $(\Gamma, A, G)$-convexity holds globally on $\partial\Omega$. 
For the local estimate in Lemma \ref{Lemma 2.2} here, since the uniform $(\Gamma, A, G)$-convexity only holds at $x_0\in \partial\Omega$, we need to make the corresponding barrier argument in a small neighbourhood $B_R(x_0)\cap \Omega$ of the boundary point $x_0$.

\begin{proof}[Proof of Lemma \ref{Lemma 2.2}]
Since the proof of this lemma is a modification of that of Lemma 2.2 in \cite{JT-oblique-I},  we only present the key steps and omit much of the calculation details.
	
For $x_0\in \partial\Omega$ and a sufficiently small positive $R$, we assume that $\varphi, \beta$ and $\nu$ in $B_R(x_0)\cap \partial\Omega$ has been smoothly extended to $\bar B_R(x_0)\cap \bar\Omega$. We consider the function
\begin{equation}
\bar v:= [D_\beta u - \varphi(\cdot, u)] + \frac{a}{2} |Du-Du(x_0)|^2, \quad {\rm in} \ B_R(x_0)\cap \Omega,
\end{equation}
where $a\le 1$ is a positive constant.
We define the linearized operator 
\begin{equation}\label{linearized op}
L:=F^{ij} [D_{ij} - A_{ij}^k (\cdot, u, Du)D_k],
\end{equation}
which is the first part of the operator $\mathcal{L}$ defined in \eqref{linearized op full}.
By (2.29) in \cite{JT-oblique-I}, in the case when F5\textsuperscript{+} holds, we can have
\begin{equation}\label{L bar v 3}
L\bar v \ge  - [\frac{C}{a}+(\epsilon_1 M_2(R) + C_{\epsilon_1})]\mathscr{T}, \quad {\rm in} \ B_R\cap \Omega,
\end{equation}
for any $\epsilon_1>0$, where Cauchy's inequality is used in the second inequality, $\mathscr{T}={\rm trace}(F^{ij})$, the constant $C$ depends on $B_R\cap \Omega, A, B, \varphi, \beta$ and $|u|_{1;B_R\cap \Omega}$, and the constant $C_{\epsilon_1}$ depends on $\epsilon_1$, $F$ and $B$. In the case when only F5 holds and $B$ is independent of $p$, the term $(\epsilon_1 M_2(R) + C_{\epsilon_1}) \mathscr{T}$ does not appear on the right hand side of the inequality \eqref{L bar v 3}.

Next, we divide into two cases.

Case (i): $D_\beta G(\cdot, u, Du) \le 0$ at $x_0$. By a direct calculation, we have
\begin{equation}
u_{\beta\beta}(x_0) \le C,
\end{equation}
where $C$ is a constant depending on $B_R\cap \Omega, \varphi, \beta$ and $|u|_{1;B_R\cap \Omega}$. Thus, we have already obtained an upper bound for pure oblique derivative of $u$ at $x_0$.

Case (ii): $D_\beta G(\cdot, u, Du) > 0$ at $x_0$. In this case, in a small neighbourhood of $x_0$, we have $G>0$ in the direction of $\beta$ at $x_0$. Then we need to construct an upper barrier function for $\bar v$ at $x_0$, using the uniform $(\Gamma, A, G)$-convexity of $\partial\Omega$ at $x_0$ with resect to $u$.
As in \cite{JT-oblique-I}, we consider the function
\begin{equation}\label{upper barrier}
\bar \phi = \phi + \frac{b}{2}|x-x_0|^2,
\end{equation}
where $\phi= c(d-td^2)$, $d=d(x)={\rm dist}(x,\partial\Omega)$, $b, c, t$ and $R$ are positive constants to be determined.
Since $\partial\Omega$ is uniformly $(\Gamma, A, G)$-convex at $x_0\in \partial\Omega$ with respect $u$, for sufficiently small $R>0$, there exists a small positive constant $\sigma$, such that
\begin{equation}\label{uniform conv neighb}
K_A[B_R\cap \partial\Omega](x,u,Du) + \mu_0 \nu(x)\otimes \nu(x) -2 \sigma I \in \Gamma,
\end{equation}
for all $x\in B_R\cap\partial\Omega$ satisfying $G(x,u(x),Du(x))\ge 0$. Here we have assumed that $|u|$ and $|Du|$ are bounded. 
Fixing the constants $R$ and $t$ such that $tR\le 1/4$, and using F4 and \eqref{uniform conv neighb}, as (2.36) in \cite{JT-oblique-I}, we have for sufficiently large $c$,
\begin{equation}\label{L barrier 2}
 L \phi \le -\frac{1}{2}c\sigma \mathscr{T}, \quad{\rm in}\ \  \{B_R\cap \Omega\}^+:= (B_R\cap \Omega) \cap \{G(\cdot,u,Du) > 0\}.
\end{equation}
Consequently, we have
\begin{equation}\label{L bar phi}
L\bar \phi \le (-\frac{1}{2}c\sigma + Cb)\mathscr{T}, \quad {\rm in} \ \ \{ B_R\cap \Omega\}^+,
\end{equation}
where the constant $C$ depends on $B_R\cap \Omega, A$ and $|u|_{1;B_R\cap \Omega}$. 
From \eqref{L bar v 3} and \eqref{L bar phi}, we now have
\begin{equation}\label{Lv ge Lphi}
L \bar \phi \le  L \bar v, \quad\quad {\rm in}\  \{B_R\cap \Omega\}^+,
\end{equation}
provided $c\ge 2[C(b+\frac{1}{a}) + (\epsilon_1 M_2(R)+C_{\epsilon_1})]/\sigma$.


Next, we examine $\bar v$ and $\bar \phi$ on the boundary of $\{ B_R\cap \Omega\}^+$. For $x\in B_{R}\cap\partial \Omega$, we have
\begin{equation}\label{compare out bdy}
\begin{array}{ll}
|Du(x)-Du(x_0)|   \!\!&\!\! \displaystyle \le (\sup_{B_R(x_0)\cap\partial\Omega}\sup_{|\tau|=1,\tau\cdot \nu=0}  |D_\tau Du| ) |x-x_0| \\
                \!\!&\!\! \displaystyle \le C (1+M^\prime_2(R))|x-x_0|,
\end{array}
\end{equation}
where the mixed derivative estimate \eqref{tan obl} and the obliqueness \eqref{ob} are used in the last inequality, so the constant $C$ depends on $B_R\cap\Omega, \varphi, \beta, \beta_0$ and $|u|_{1;B_R\cap\Omega}$. 
Then we have from \eqref{compare out bdy}
\begin{equation}\label{boundary inequality 1}
\bar v  \le \frac{a}{2} C [1+ (M_2^\prime(R))^2]|x-x_0|^2 \le \bar \phi, \quad {\rm on} \ B_R\cap \partial\Omega,
\end{equation}
provided $b\ge aC[1+ (M_2^\prime(R))^2]$, for a further constant $C$.
For $x\in B_R\cap\Omega$ and $x^\prime$ the closest point on $B_R\cap\partial \Omega$, we then obtain,
\begin{equation}
\begin{array}{ll}
|Du(x)-Du(x_0)|^2 \!\!&\!\! \displaystyle \le 4(\sup\limits_{B_R\cap \Omega}|Du|)M_2(R)d(x) + 2|Du(x^\prime)-Du(x_0)|^2\\
                \!\!&\!\! \displaystyle \le C [1+(M^\prime_2(R))^2 + M_2(R)] (|x-x_0|^2 +d),
\end{array}
\end{equation}
so that
\begin{equation}\label{boundary inequality 2}
\bar v \le \frac{1}{2}a C [1+(M^\prime_2(R))^2 + M_2(R)](|x-x_0|^2 + d) \le \bar \phi, \quad {\rm on} \ (B_R\cap\Omega)\cap\{G(\cdot,u,Du) = 0\},
\end{equation}
by choosing $b\ge aC [1+(M^\prime_2(R))^2 +M_2(R)]$ and $c \ge b$.
On the inner boundary, we have
\begin{equation}\label{boundary inequality 3}
\begin{array}{rl}
\bar v \le \!\!&\!\! \displaystyle \sup\limits_{B_R\cap\Omega}  \{[D_\beta u - \varphi(\cdot, u)] + \frac{a}{2} |Du-Du(x_0)|^2\} \\
          \le \!\!&\!\! \displaystyle \frac{b}{2}R^2 \le cd(1-tR)+\frac{b}{2}R^2 \le \bar \phi, \quad {\rm on} \ \partial B_R\cap \Omega,
\end{array}
\end{equation}
provided $b\ge \frac{C}{R^2}$, for the constant $C$ depending on $B_R\cap\Omega, \varphi, \beta$ and $|u|_{1;B_R\cap \Omega}$.

Now from \eqref{Lv ge Lphi}, \eqref{boundary inequality 1}, \eqref{boundary inequality 2} and \eqref{boundary inequality 3}, by the comparison principle, we have
\begin{equation}
\bar v \le \bar \phi, \quad {\rm in} \ B_R\cap \Omega.
\end{equation}
Since $\bar v(x_0)= \bar \phi(x_0)=0$, we have $D_\beta \bar v(x_0) \le D_\beta \bar \phi(x_0)$, which implies
\begin{equation}\label{double oblique at x0}
u_{\beta\beta} (x_0) \le \beta^0 c+C,
\end{equation}
where $\beta^0:=\sup\limits_{\partial\Omega}(\beta \cdot \nu)\ge \beta_0$.
In view of the above considerations, we can fix the constant $b=aC[1+(M_2^\prime(R))^2+M_2(R)]+ \frac{C}{R^2}$.
We can further fix the constant $c$ so that
\begin{equation}\label{choice of c}
\begin{array}{ll}
c \!\!&\!\! \displaystyle = \frac{2[C(b+\frac{1}{a})+(\epsilon_1 M_2(R)+C_{\epsilon_1}) ]}{\sigma} + b\\
  \!\!&\!\! \displaystyle \le C [(\epsilon_1+a)M_2(R)+a(M^\prime_2(R))^2 +\frac{1}{a}] + C_{\epsilon_1} +\frac{C^\prime}{R^2},
\end{array}
\end{equation}
where $C$ now depends on $B_R\cap \Omega, F, A, B, \varphi, \beta, \beta_0$ and $|u|_{1;B_R\cap\Omega}$, $C_{\epsilon_1}$
depends additionally on $\epsilon_1$, and $C^\prime$ depends on $B_R\cap \Omega, A, \varphi, \beta$ and $|u|_{1;B_R\cap\Omega}$. 
For any $\epsilon >0$, taking $a=\frac{1}{1+\epsilon_1 M_2(R)}$ and $\epsilon_1= \frac{\epsilon}{\beta^0C}$ for a further constant $C$ in \eqref{choice of c}, from \eqref{double oblique at x0} and \eqref{choice of c} we arrive at the estimate \eqref{local pure obl est}, for both F5\textsuperscript{+} and $B$ independent of $p$.

We have proved this lemma in the case when $R$ is sufficiently small. When $R$ is larger, we can get through by using the same argument at the end of the proof of Lemma \ref{Lemma 2.1}.
\end{proof}

\begin{remark}\label{uniform convexity remark}
In view of \eqref{uniform conv neighb}, when $\partial\Omega$ is uniformly $(\Gamma, A, G)$-convex at $x_0\in \partial\Omega$, with respect to $u$,  there exists a sufficiently small $R>0$ such that $\partial\Omega$ is uniformly $(\Gamma, A, G)$-convex at each $x\in B_{R}(x_0)\cap\partial\Omega$, with respect to $u$. Therefore, the estimate \eqref{local pure obl est} will hold for all the points $x\in B_{R}(x_0)\cap\partial\Omega$. 
\end{remark}

By making full use of the local/global second derivative estimate (1.14) in \cite{JT-oblique-I} together with the local boundary estimates \eqref{tan obl}, \eqref{local pure prime estimate} and \eqref{local pure obl est} in the mixed tangential-oblique, pure tangential and pure oblique directions respectively, we are now able to establish the following local second derivative estimates for the boundary value problem \eqref{1.1}-\eqref{1.5}.

\begin{Theorem}\label{Th2.1}
Let $u\in C^2(\bar \Omega)\cap C^4(\Omega)$ be an admissible solution of equation \eqref{1.1} in a bounded domain $\Omega\subset \mathbb{R}^n$. Assume that $F\in C^2(\Gamma)$ satisfies conditions F1-F3 and F5 in the cone $\Gamma\subset \mathbb{S}^n$, $A\in C^2(\bar \Omega \times \mathbb{R} \times \mathbb{R}^n)$ is strictly regular, $B>a_0, \in C^2(\bar\Omega\times \mathbb{R} \times \mathbb{R}^n)$, either F5\textsuperscript{+} holds or $B$ is convex with respect to $p$. 
For any point $x_0\in \bar\Omega$ and positive constant $R>0$, we have
\begin{equation}\label{est LG}
M_2(\theta R) \le \tilde M_2 + C(1+\frac{1}{R^2}),
\end{equation}
for any constant $0<\theta <1$, where $M_2(\theta R)=\sup\limits_{B_{\theta R}\cap\Omega}|D^2u|$, $\tilde M_2 =\sup\limits_{B_{R}\cap\partial\Omega}|D^2u|$, with $B_R=B_R(x_0)$, and the constant $C$ depends on $\theta, B_R\cap \Omega, \Gamma, F, A, B$ and $|u|_{1;B_R\cap \Omega}$. Assume in addition that $\partial\Omega\in C^{2,1}$, $\Gamma\subset \mathcal{P}_{n-1}$, $G\in C^2(\partial\Omega\times \mathbb{R}\times \mathbb{R}^n)$ is semilinear and oblique satisfying \eqref{1.5} and \eqref{ob} and either F5\textsuperscript{+} holds or $B$ is independent of $p$. If F4 holds and $\partial\Omega$ is uniformly $(\Gamma, A, G)$-convex at $x_0\in \partial\Omega$, with respect to $u$, then there exists a sufficiently small $R>0$, such that
\begin{equation}\label{est int bdy}
M_2(\theta R) \le C(1+\frac{1}{R^2}),
\end{equation}
for any constant $0<\theta <1$, where the constant $C$ depends on $\theta, B_R\cap \Omega, F, A, B, \varphi, \beta, \beta_0$ and $|u|_{1;B_R\cap \Omega}$.
\end{Theorem}

\begin{proof}
Under the assumptions for \eqref{est LG}, the local/global estimate (1.14) in Theorem 1.1 in \cite{JT-oblique-I} holds, namely
\begin{equation}\label{LG}
\sup_{\Omega\cap \Omega^\prime} |D^2u| \le \sup_{\partial\Omega\cap \Omega_0} |D^2u| + C,
\end{equation}
for any domains $\Omega^\prime \subset\subset \Omega_0 \subset \mathbb{R}^n$.
If we choose $\Omega_0=B_R(x_0)$, $\Omega^\prime = B_{\theta R}(x_0)$, then the estimate \eqref{est LG} is a direct consequence of \eqref{LG}, where the term $C(1+\frac{1}{R^2})$ is from the differentiation of the cut-off function $\zeta$. In particular we may choose a cut-off function $\zeta \in C_0^2(B_R(x_0))$, such that $0\le \zeta \le 1$ in $B_R(x_0)$, $\zeta =1$ in $B_{\theta R}(x_0)$, and  $|D\zeta|^2 + |D^2\zeta| \le \frac{C}{(1-\theta)^2R^2}$ for a universal constant $C$.  

From \eqref{est LG}, if $B_R(x_0)\subset \Omega$, we get an interior second derivative estimate
\begin{equation}\label{int est in Th2.1}
M_2(\theta R) \le C(1+\frac{1}{R^2}),
\end{equation} 
for any $0<\theta <1$, since $M_2^\prime (R) = 0$. 
Next, when $\Gamma\subset \mathcal{P}_{n-1}$, F4 holds and $\partial\Omega$ is uniformly $(\Gamma, A, G)$-convex at $x_0\in \partial\Omega$, with respect to $u$, we take full advantage of the estimates \eqref{tan obl}, \eqref{local pure prime estimate}, \eqref{local pure obl est} and \eqref{est LG} to derive the desired estimate \eqref{est int bdy}.
From Remark \ref{uniform convexity remark}, there exists a small $R>0$ such that $\partial\Omega$ is uniformly $(\Gamma, A, G)$-convex at each $x\in B_{R}(x_0)\cap\partial\Omega$, with respect to $u$.
Set 
\begin{equation}
M_2^*(R) = \sup_{B_R(x_0) \cap \Omega} (d_x^2 |D^2u(x)|),
\end{equation}
where $d_x:={\rm dist}(x, \partial B_R(x_0))$. In order to get the estimate for $M_2(\theta R)$ in \eqref{est int bdy}, we first obtain an estimate for $M_2^*(R)$ .
We assume that $M_2^*(R)$ is attained at a point $\bar x \in B_R(x_0)\cap \bar \Omega$, namely
\begin{equation}
M_2^*(R) = d_{\bar x}^2 |D^2u (\bar x)| = (\delta R)^2|D^2u (\bar x)|,
\end{equation}
where $\delta$ is a constant in $(0, 1]$ such that $\delta={d_{\bar x}}/{R}$. We divide into two cases according to the positions of $\bar x$.

Case (i): $d(\bar x) \ge \frac{d_{\bar x}}{8}$, where $d(\bar x):={\rm dist}(\bar x, \partial\Omega)$. In this case, we have $B_{\frac{\delta R}{16}}(\bar x) \subset B_{\frac{\delta R}{8}}(\bar x) \subset B_R(x_0) \cap \Omega$. Then we have
\begin{equation}\label{up bd M* int}
\begin{array}{rl}
M_2^*(R) = \!\!&\!\!\displaystyle \sup\limits_{B_{\frac{\delta R}{16}}(\bar x)} (d_x^2 |D^2u(x)|) \le (\frac{17}{16}\delta R)^2\sup\limits_{B_{\frac{\delta R}{16}}(\bar x)} |D^2u| \\
              \le \!\!&\!\!\displaystyle (\frac{17}{16}\delta R)^2 C (1+\frac{1}{(\delta R/ 8)^2}),
\end{array}
\end{equation}
where the interior estimate \eqref{int est in Th2.1} with $\theta=1/2$ is used in the last inequality. On the other hand, we have
\begin{equation}\label{lw bd M* int}
M_2^*(R) \ge [(1-\theta)R]^2 M_2(\theta R),
\end{equation}
for any $0<\theta <1$. Combining \eqref{up bd M* int} and \eqref{lw bd M* int}, we obtain the desired estimate \eqref{est int bdy}.

Case (ii): $d(\bar x) < \frac{d_{\bar x}}{8}$. Let $\bar x^\prime$ be the nearest point of $\bar x$ to $\partial\Omega$. (Note that if $\bar x\in \partial\Omega$, then $\bar x^\prime =\bar x$ and $d(\bar x)=0$.) Since $d_{\bar x}=\delta R$, $d(\bar x)<\frac{\delta R}{8}$, we have $\frac{7}{8}\delta R<d_{\bar x^\prime}<\frac{9}{8}\delta R$ and $\frac{5}{8}\delta R<{\rm dist} (B_{\frac{\delta R}{4}}(\bar x^\prime)\cap \Omega, \partial B_R(x_0))< \frac{11}{8}\delta R$. Then we have $\bar x \in B_{2d(\bar x)}(\bar x^\prime)\subset B_{\frac{\delta R}{4}}(\bar x^\prime) \subset B_R(x_0)$, and
\begin{equation}\label{up bd M* bdy}
\begin{array}{rl}
M_2^*(R) = \!\!&\!\!\displaystyle \sup\limits_{B_{\frac{\delta R}{4}}(\bar x^\prime)\cap \Omega} (d_x^2 |D^2u(x)|) \le (\frac{11}{8}\delta R)^2\sup\limits_{B_{\frac{\delta R}{4}}(\bar x^\prime)\cap \Omega} |D^2u| \\
              \le \!\!&\!\!\displaystyle (\frac{11}{8}\delta R)^2 [\sup\limits_{B_{\frac{\delta R}{2}}(\bar x^\prime)\cap \partial\Omega} |D^2u|+ C (1+ \frac{1}{(\delta R/2)^2})],
\end{array}
\end{equation}
where the estimate \eqref{est LG} is used with $\theta=1/2$ and $B_R(x_0)$ replaced by $B_{\frac{\delta R}{4}}(\bar x^\prime)$. Next, we need to derive an estimate for $\sup\limits_{B_{\frac{\delta R}{2}}(\bar x^\prime)\cap \partial\Omega} |D^2u|$, using Lemmas \ref{Lemma 2.1} and \ref{Lemma 2.2}. Since ${\rm dist} (B_{\frac{\delta R}{2}}(\bar x^\prime)\cap \Omega, \partial B_R(x_0))> \frac{3}{8}\delta R$, then for any point $y\in B_{\frac{\delta R}{2}}(\bar x^\prime)\cap \partial\Omega$, we have $B_{\frac{\delta R}{4}}(y) \subset B_R(x_0)$ and ${\rm dist} (B_{\frac{\delta R}{4}}(y)\cap \Omega, \partial B_R(x_0))> \frac{1}{8}\delta R$. Applying estimate \eqref{local pure obl est} in Lemma \ref{Lemma 2.2} to the ball $B_{\frac{\delta R}{4}}(y)$, we have
\begin{equation}\label{apply pure obl}
u_{\beta\beta} (y) \le \epsilon \sup\limits_{B_{\frac{\delta R}{4}}(y)\cap \Omega}|D^2u| + C_\epsilon (1+ \sup\limits_{B_{\frac{\delta R}{4}}(y)\cap \partial\Omega}\sup\limits_{|\tau|=1,\tau\cdot \nu=0} |u_{\tau\tau}|) + \frac{C}{(\delta R/ 4)^2},
\end{equation}
for any $y\in B_{\frac{\delta R}{2}}(\bar x^\prime)\cap \partial\Omega$, and $\epsilon >0$. Since $B_{\frac{\delta R}{4}}(y) \subset B_{\frac{3\delta R}{4}}(\bar x^\prime)$ for all $y\in B_{\frac{\delta R}{2}}(\bar x^\prime)\cap \partial\Omega$, we have from \eqref{apply pure obl},
\begin{equation}\label{apply pure obl 1}
\sup\limits_{B_{\frac{\delta R}{2}}(\bar x^\prime)\cap \partial\Omega} u_{\beta\beta} \le \epsilon \sup\limits_{B_{\frac{3\delta R}{4}}(\bar x^\prime)\cap \Omega}|D^2u| + C_\epsilon (1+ \sup\limits_{B_{\frac{3\delta R}{4}}(\bar x^\prime)\cap \partial\Omega}\sup\limits_{|\tau|=1,\tau\cdot \nu=0} |u_{\tau\tau}|) + \frac{C}{(\delta R/ 4)^2},
\end{equation}
for any $\epsilon>0$, where $B_{\frac{3\delta R}{4}}(\bar x^\prime)\subset B_R(x_0)$, and ${\rm dist}(B_{\frac{3\delta R}{4}}(\bar x^\prime)\cap \Omega, \partial B_R(x_0))> \frac{1}{8}\delta R$.
For any point $\tilde y\in B_{\frac{3\delta R}{4}}(\bar x^\prime)\cap \partial\Omega$, we have $B_{\frac{\delta R}{16}}(\tilde y)\subset B_R(x_0)$, and ${\rm dist} (B_{\frac{3\delta R}{4}}(\bar x^\prime)\cap\Omega, \partial B_{R}(x_0))>\frac{1}{16}\delta R$.
Applying estimate \eqref{local pure prime estimate} in Lemma \ref{Lemma 2.1} in the ball $B_{\frac{\delta R}{16}}(\tilde y)$, we have
\begin{equation}\label{apply pure tan}
\sup\limits_{|\tau|=1,\tau\cdot \nu=0} |u_{\tau\tau}(\tilde y)| \le \epsilon \sup\limits_{B_{\frac{\delta R}{16}}(\tilde y)\cap\Omega} |D^2u| + C_\epsilon (1+ \frac{1}{(\delta R/16)^2}),
\end{equation}
for any $\tilde y\in B_{\frac{3\delta R}{4}}(\bar x^\prime)\cap \partial\Omega$, and $\epsilon>0$. Since $B_{\frac{\delta R}{16}}(\tilde y) \subset B_{\frac{13\delta R}{16}}(\bar x^\prime)$ for all $\tilde y\in B_{\frac{3\delta R}{4}}(\bar x^\prime)\cap \partial\Omega$, we have from \eqref{apply pure tan},
\begin{equation}\label{apply pure tan 1}
\sup\limits_{B_{\frac{3\delta R}{4}}(\bar x^\prime)\cap \partial\Omega}\sup\limits_{|\tau|=1,\tau\cdot \nu=0} |u_{\tau\tau}| \le \epsilon \sup\limits_{B_{\frac{13\delta R}{16}}(\bar x^\prime)\cap\Omega} |D^2u| + C_\epsilon (1+ \frac{1}{(\delta R/16)^2}),
\end{equation}
for any $\epsilon>0$, where $B_{\frac{13\delta R}{16}}(\bar x^\prime)\subset B_R(x_0)$ and ${\rm dist} (B_{\frac{13\delta R}{16}}(\bar x^\prime)\cap\Omega, \partial B_R(x_0))>\frac{\delta R}{16}$. Plugging \eqref{apply pure tan 1} into \eqref{apply pure obl 1}, we obtain
\begin{equation}\label{apply pure obl 2}
\sup\limits_{B_{\frac{\delta R}{2}}(\bar x^\prime)\cap \partial\Omega} u_{\beta\beta} \le \epsilon \sup\limits_{B_{\frac{13\delta R}{16}}(\bar x^\prime)\cap \Omega}|D^2u| + C_\epsilon (1+ \frac{1}{(\delta R)^2} ),
\end{equation}
for any $\epsilon >0$ and a further constant $C_\epsilon$. From \eqref{apply pure tan 1}, \eqref{apply pure obl 2} and \eqref{tan obl}, we have
\begin{equation}\label{up full bd}
\sup\limits_{B_{\frac{\delta R}{2}}(\bar x^\prime)\cap \partial\Omega} u_{\xi\xi} \le \epsilon \sup\limits_{B_{\frac{13\delta R}{16}}(\bar x^\prime)\cap \Omega}|D^2u| + C_\epsilon (1+ \frac{1}{(\delta R)^2} ),
\end{equation}
for any unit vector $\xi$, and any $\epsilon>0$. 
Using the concavity condition F2 or $\Gamma\subset \mathcal{P}_{n-1}$ as in \cite{JT-oblique-I}, from the upper bound in \eqref{up full bd}, we have
\begin{equation}\label{up lw full bd}
\sup\limits_{B_{\frac{\delta R}{2}}(\bar x^\prime)\cap \partial\Omega} |D^2u| \le \epsilon \sup\limits_{B_{\frac{13\delta R}{16}}(\bar x^\prime)\cap \Omega}|D^2u| + C_\epsilon (1+ \frac{1}{(\delta R)^2} ),
\end{equation}
for any $\epsilon >0$. Since ${\rm dist}(B_{\frac{13\delta R}{16}}(\bar x^\prime)\cap \Omega, \partial B_R(x_0))> \frac{\delta R}{16}$, we have
\begin{equation}\label{theta relation}
B_{\frac{13\delta R}{16}}(\bar x^\prime) \subset B_{\theta R}(x_0),
\end{equation}
for $\theta \in (1-\frac{\delta}{16}, 1)$. Plugging \eqref{up lw full bd}, \eqref{theta relation} into \eqref{up bd M* bdy}, we obtain
\begin{equation}\label{up bd M* bdy 1}
M_2^*(R) \le (\frac{11}{8}\delta R)^2 [\epsilon M_2(\theta R) + C_\epsilon (1+ \frac{1}{(\delta R)^2})],
\end{equation}
for any $\epsilon>0$, and $\theta \in (1-\frac{\delta}{16}, 1)$, and a further constant $C_\epsilon$. Combining \eqref{lw bd M* int} and \eqref{up bd M* bdy 1}, and choosing a sufficiently small $\epsilon>0$, we obtain the estimate \eqref{est int bdy} for $\theta \in (1-\frac{\delta}{16}, 1)$. Since $M_2(\theta R)$ is nondecreasing in $\theta$,  \eqref{est int bdy}  holds for any $\theta \in (0,1)$. 

From the above two cases (i) and (ii), we have thus completed the proof of the estimate \eqref{est int bdy}.
\end{proof}

\begin{remark}\label{Rem 2.2}
Note that by plugging \eqref{est int bdy} into \eqref{up bd M* bdy 1}, we  thereby obtain the estimate
\begin{equation}\label{weighted norm est}
M_2^* (R) \le C(1+R^2), 
\end{equation}
for a further constant $C$ in case (ii). Actually, \eqref{up bd M* int} gives the estimate \eqref{weighted norm est} in case (i). Therefore, the estimate \eqref{weighted norm est} holds under the assumptions of Theorem \ref{Th2.1}.
\end{remark}

\begin{remark}\label{Rem 2.3}
In Theorem \ref{Th2.1}, the local estimate \eqref{est int bdy} is obtained in a small ball $B_R(x_0)$ with centre $x_0\in \partial\Omega$. In fact, in any small ball $B_{\bar R}(y)$  with $B_{\bar R}(y)\cap \partial\Omega\subseteq B_R(x_0)\cap\partial\Omega$,  \eqref{est int bdy} still holds with $M_2(\theta R)$ and $R$ replaced by $\sup\limits_{B_{\theta \bar R}(y)\cap \Omega}|D^2u|$ and $\bar R$ respectively.
\end{remark}

\begin{remark}\label{Rem 2.4}
For the local estimates in this section, F1 and the condition $B > a_0$ can be replaced by the weaker assumptions F1$^-$ and $B\ge a_0$ if $a_0 > -\infty$ and F5$^+$ is strengthened to F5$^+$(0), but it is not feasible to consider solutions with smooth second derivatives in these degenerate cases. In later sections, we can directly use these local second derivative estimates, when considering the smooth solutions of the regularized problem under F1$^-$ or the weak inequality $B \ge a_0$.
\end{remark}

\begin{remark}\label{Rem 2.5}
Following Remark \ref{Rem 1.4}, we can replace the cone $\Gamma$ in the estimates of this section by any convex open set $D\subset \mathbb{S}^n$, $\ne \mathbb{S}^n$, with $0\in\partial D$, which is closed under addition of the positive cone, $K^+$, provided its asymptotic cone $\Gamma$ is used in our hypotheses of ($\Gamma, A, G$) -convexity and in condition F4. Note that in Lemma \ref{Lemma 2.2} and Theorem \ref{Th2.1}, F4 is always assumed in company with the uniform $(\Gamma, A, G)$-convexity of a boundary point and is automatically satisfied if $F$ satisfies F1$^-$, F2 and F3 when $a_0>-\infty$, while F5(0) is satisfied if $D = \Gamma$. For general $D$ we thus have to also strengthen both F5 to F5(0) and  F5$^+$  to F5$^+$(0) in the degenerate case $B\ge a_0$ for the validity of Lemmas \ref{Lemma 2.1}, \ref{Lemma 2.2} and Theorem \ref{Th2.1}. 
\end{remark}

\vspace{3mm}

\section{Existence}\label{Section 3}

\vskip10pt

In this section, following \cite{JT-oblique-I}, we introduce a uniformly elliptic regularization of the oblique derivative problem \eqref{1.1}-\eqref{1.5} and prove its classical solvability. By using the regularized problems for approximations, we prove in Theorem \ref{Th3.1} the existence of solutions in $C^{1,1}(\Omega)\cap C^{0,1}(\bar\Omega)$,  of the original oblique derivative problem \eqref{1.1}-\eqref{1.5} under F1$^-$, which lie in  $C^{2,\alpha}(\Omega)$ for some $\alpha\in (0,1)$ in the nondegenerate case when F1 holds and $B> a_0$. Note that we only assume that the domain $\Omega$ is a bounded open set in $\mathbb{R}^n$ without any convexity assumptions on the boundary. At the end of this section, we also consider alternative hypotheses for the gradient estimates, which can lead to more general versions of the existence results in Theorems \ref{Th1.1} and Theorem \ref{Th3.1}.

We first recall the definition of uniform ellipticity of fully nonlinear elliptic equations from \cite{LieTru1986}. Letting $\lambda$, $\Lambda$ denote the minimum and maximum eigenvalues of $F_r$, we  call $F$ uniformly elliptic with respect to some subset $\mathcal{U}\subset \Gamma$, if the ratio $\Lambda/ \lambda$ is bounded on $\mathcal{U}$. Following \cite{JT-oblique-I}, our regularization is achieved by adding 
\begin{equation}
\epsilon F_1(M[u]) I
\end{equation}
to the augmented Hessian matrix $M[u]$, where $\epsilon> 0$ and $F_1(M[u])={\rm trace}(M[u])$.
For the operator $F$ and cone $\Gamma$, we define for a constant $\epsilon> 0$, $F_1(r)={\rm trace}(r)$, the approximating operators and approximating cones,
\begin{equation}\label{uni reg}
F^\epsilon (r) = F(r+\epsilon F_1(r)I), \quad \Gamma^\epsilon = \{ r+\epsilon F_1(r)I | \  r\in \Gamma \}.
\end{equation}
Assuming the cone $\Gamma\subset \Gamma_1$ without loss of generality, we have $F_1(r)>0$ for $r\in \Gamma$. Then $r+\epsilon F_1(r)I\in \Gamma$ for $\epsilon> 0$, if $r\in \Gamma$. Thus, we have 
\begin{equation}
\Gamma^\epsilon \subset \Gamma,
\end{equation}
for $\epsilon > 0$. By a calculation, we have $F_r^\epsilon = F_r + \epsilon \mathscr{T} I$, where $\mathscr{T}=\mathscr{T}(r)$. Therefore, we have
\begin{equation}\label{uniform ellipticity}
\epsilon \mathscr{T} I \le F_r^\epsilon(r) \le (1+\epsilon) \mathscr{T} I.
\end{equation}
The ratio of the maximum and minimum eigenvalues of $F_r^\epsilon$ is $1+\frac{1}{\epsilon}$, which is bounded for any fixed positive constant $\epsilon$. Then the operator $F^\epsilon$ is uniformly elliptic with respect to $\Gamma$ for any fixed positive constant $\epsilon$. Moreover, as in Section 4.3 in \cite{JT-oblique-I}, we know that if $F$ satisfies any of the conditions F1 to F5 (F5\textsuperscript{+}), then $F^\epsilon$ satisfies the same conditions as $F$ with further constants independent of $\epsilon$.

Now, we introduce the regularized problem of the oblique derivative problem \eqref{1.1}-\eqref{1.5},
\begin{equation}\label{re-problem}
\left \{
\begin{array}{rll}
\mathcal{F}^\epsilon [u] \!\!&\!\! = B(\cdot, u, Du), \!\!&\!\! \quad {\rm in} \ \Omega,\\
\mathcal{G}[u] \!\!&\!\! =0, \!\!&\!\! \quad {\rm on} \ \partial\Omega,
\end{array}
\right.
\end{equation}
for $\epsilon >0$, where $\mathcal{F}^\epsilon [u]:= F^\epsilon(M[u])$, $\mathcal{G}[u]=\beta\cdot Du - \varphi(\cdot, u)$. We will first study the existence result for the regularized problem \eqref{re-problem} for $\epsilon>0$, and then send $\epsilon$ to zero to get the existence result for the original problem \eqref{1.1}-\eqref{1.5}.


In order to establish the existence of the classical admissible solutions for the regularized problem \eqref{re-problem}, we need to obtain the {\it a priori} solution estimates and derivative estimates for the solution $u_{\epsilon}$ of \eqref{re-problem}. In the estimates discussed in this part, we mainly focus on the dependence on the constant $\epsilon$ for various derivative bounds. We shall discuss these estimates under the assumptions in Theorem \ref{Th1.1}(i), unless otherwise specified. Note that under F1$^-$, the second derivative estimate and the gradient estimate in \cite{JT-oblique-I} are still valid, (see the paragraph above Corollary 4.1 in \cite{JT-oblique-I}).

When $A, B$ and $\varphi$ are nondecreasing in $z$, with either $A$ or $B$ strictly increasing, as in (1.27) in \cite{JT-oblique-I}, the function 
\begin{equation}\
\bar u := c_0 + c_1\psi
\end{equation}
will be a supersolution of \eqref{1.1}-\eqref{1.5}, where $c_0$ and $c_1$ are sufficiently large constants, $\psi\in C^2(\bar \Omega)$ is a defining function for $\Omega$ satisfying $\psi<0$ in $\Omega$, $\psi=0$ and $D\psi \neq 0$ on $\partial\Omega$. When only $\varphi$ is strictly increasing in $z$, the function 
\begin{equation}
\bar u + c_2 
\end{equation}
will be a supersolution of \eqref{1.1}-\eqref{1.5}, where $\bar u$ is a supersolution of \eqref{1.1} as assumed in Theorem \ref{Th1.1}, $c_2$ is a sufficiently large positive constant. For convenience, we shall use $\bar u$ to denote the supersolution of \eqref{1.1}-\eqref{1.5} in all the above cases.
Since we already have a strict subsolution $\underline u\in  C^{1,1}(\bar \Omega)$ and a supersolution $\bar u \in  C^{1,1}(\bar \Omega)$ of \eqref{1.1}-\eqref{1.5} in hand, following the proof of Corollary 4.1 in \cite{JT-oblique-I}, we obtain uniform  lower and upper bounds of $u_\epsilon$,
\begin{equation}\label{uni sol}
\underline u \le u_\epsilon \le \bar u, \quad {\rm in} \  \bar \Omega, 
\end{equation}
for sufficiently small $\epsilon>0$. Under F1$^-$, F2, F3 and F5($\infty$), since $A$ is uniformly regular in $\Omega$ and the quadratic growth conditions \eqref{quadratic growth} hold, from Theorem 1.3 or Lemma 3.1 in \cite{JT-oblique-I}, we have the uniform gradient estimate of $u_\epsilon$,
\begin{equation}\label{uni gra}
\sup_{\Omega}|Du_\epsilon| \le C,
\end{equation}
where $C$ is a constant depending on $F, A, B, \Omega, \beta, \varphi$ and $|u_\epsilon|_{0,\Omega}$, and is independent of $\epsilon$. Since when $a_0=0$, F1$^-$, F2 and F3 imply F5(0), so that the constant $C$ in \eqref{uni gra} does not depend on $b_0:=\inf_{\Omega}B$ and the estimate \eqref{uni gra} thus holds for $B\ge 0$. Then using the uniform ellipticity together with the estimates \eqref{uni sol} and \eqref{uni gra}, we obtain the second derivative estimate of $u_{\epsilon}$ by Theorem 5.4 in \cite{LieTru1986}, 
\begin{equation}\label{2nd bd epsilon}
\sup_{\Omega} |D^2 u_\epsilon| \le C_\epsilon,
\end{equation}
where $C_\epsilon$ is a constant depending on $\epsilon$ and other known data. Once we have the full second order derivative bound \eqref{2nd bd epsilon}, we can use the uniformly elliptic theory as in Theorem 3.2 in \cite{LT1986} or Theorem 1.1 in \cite{LieTru1986} to derive the global second order derivative H{\" o}lder estimate
\begin{equation}\label{2nd Holder bd}
|u_\epsilon|_{2,\gamma; \Omega} \le C_\epsilon,
\end{equation}
for any $\gamma\in (0,1)$, where the constant $C_\epsilon$ depends on $\epsilon$ and other known data. 

With the $C^{2,\gamma}$ estimate \eqref{2nd Holder bd}, we can use the method of continuity as in Theorem 17.28 in \cite{GTbook}  or Corollary 1.2 in \cite{LieTru1986} to get the existence of a unique admissible solution $u_\epsilon \in C^{2,\gamma}(\bar \Omega)$, for any $\gamma\in (0,1)$ and any small constant $\epsilon>0$.

\begin{remark}\label{Rem 3.1}
The second derivative estimate \eqref{2nd bd epsilon} for the regularized problem \eqref{re-problem} is obtained directly from the uniform elliptic theory in \cite{LieTru1986}. Under the assumptions of Theorem \ref{Th1.1}, we can also use the estimates in \cite{JT-oblique-I} to derive the second derivative estimate \eqref{2nd bd epsilon}, by combining the local/global estimate in Theorem 1.1 in \cite{JT-oblique-I}, mixed tangential-oblique boundary estimate (2.22) in \cite{JT-oblique-I}, pure tangential boundary estimate in Lemma 2.3 in \cite{JT-oblique-I} and the pure normal boundary estimate from the uniform ellipticity. Note that for the estimates of $u_\epsilon$, we need to replace the linearized operators $\mathcal L$ in \eqref{linearized op full} and $L$ in \eqref{linearized op} by
\begin{equation}
\mathcal{L}^\epsilon : = L^\epsilon - D_{p_k}B(\cdot, u_\epsilon, Du_\epsilon) D_k,
\end{equation}
and
\begin{equation}
L^\epsilon := (F^\epsilon)^{ij} [(D_{ij}-A_{ij}^k (\cdot, u_\epsilon, Du_\epsilon) D_k) + \epsilon \delta_{ij} \sum_{l=1}^n (D_{ll}-A_{ll}^k(\cdot, u_\epsilon, Du_\epsilon)D_k)] ,
\end{equation}
respectively, where $(F^\epsilon)^{ij} = \frac{\partial F^\epsilon}{\partial r_{ij}}$.
\end{remark}

By letting $\epsilon \rightarrow 0$ in \eqref{re-problem}, we obtain the existence result for the boundary value problem \eqref{1.1}-\eqref{1.5}. We cover both the F1$^-$ case and the F1 case in the following theorem.

\begin{Theorem}\label{Th3.1}
Assume that $F$ satisfies conditions F1$^-$ (F1), F2 and F3, with $a_0 = 0$,  in $\Gamma$, $\Omega$ is a $C^{2,1}$ bounded domain in $\mathbb{R}^n$, $A\in C^2(\bar \Omega\times \mathbb{R}\times \mathbb{R}^n)$ is uniformly regular in $\Omega$, $B \ge 0, (>0), \in C^2(\bar \Omega\times \mathbb{R}\times \mathbb{R}^n)$, 
$\mathcal G$ is semilinear and oblique with $\beta\in C^{1,1}(\partial\Omega)$, $\varphi \in C^{1,1}(\partial\Omega\times \mathbb{R})$ and there exist a strict (non-strict) subsolution $\underline u\in  C^{1,1}(\bar \Omega)$ of \eqref{1.1}-\eqref{1.5} and a supersolution $\bar u \in  C^{1,1}(\bar \Omega)$ of \eqref{1.1}. Assume also that $A$ and $B$ satisfy the quadratic growth conditions \eqref{quadratic growth}, $A$, $B$ and $\varphi$ are nondecreasing with respect to $z$, with one of them strictly increasing and either (a) B is independent of $p$  or (b) $B$ is convex in $p$ and $F$ satisfies F5\textsuperscript{+}(0) (F5\textsuperscript{+}) and F5($\infty$). 
Then there exists an admissible solution $u\in C^{1,1}(\Omega)\cap C^{0,1}(\bar\Omega)$, ($u \in C^{2, \alpha}(\Omega)\cap C^{0,1}(\bar \Omega)$ for some $\alpha\in (0,1)$), of the boundary value problem \eqref{1.1}-\eqref{1.5}. 
\end{Theorem}

\begin{proof}
Assume first that $F\in C^2(\Gamma)$. We consider the case when $F$ satisfies F1$^-$ and $B\ge 0$.
We have already proved that there exists an admissible solution $u_\epsilon \in C^{2,\gamma}(\bar \Omega)$ for any $\gamma\in (0,1)$ of the regularized problem \eqref{re-problem} for any small $\epsilon>0$. From \eqref{uni sol} and \eqref{uni gra}, we have the uniform estimate
\begin{equation}\label{uni gra estimate}
|u_\epsilon|_{1;\Omega} \le C,
\end{equation}
with the constant $C$ independent of $\epsilon$.
Since $A$ is uniformly regular in $\Omega$ and $B$ satisfies either (a) or (b), taking $\Omega_0=\Omega$ in the local/global estimate in Theorem 1.1 in \cite{JT-oblique-I}, we have
\begin{equation}\label{uni int 2nd}
\sup_{\Omega^\prime} |D^2u_\epsilon| \le C,
\end{equation}
for any $\Omega^\prime \subset\subset \Omega$, where the constant $C$ is independent of $\epsilon$.   
Hence, from the uniform estimates \eqref{uni gra estimate} and \eqref{uni int 2nd}, there exists a subsequence $\{u_{\epsilon_k}\}$ and a function $u\in C^{1,1}(\Omega) \cap C^{0,1}(\bar \Omega)$ such that
\begin{equation}\label{epsilon tends 0}
u_{\epsilon_k} \rightarrow u, \quad {\rm in} \ C^{1,\alpha_1}(\Omega)\cap C^{0,\alpha_2}(\bar \Omega), 
\end{equation}
for all $\alpha_1, \alpha_2 \in (0, 1)$, as $\epsilon_k \rightarrow 0$. From the stability property of viscosity solutions \cite{CIL1992, E1978}, it is readily seen that $u$ is an admissible solution of the problem \eqref{1.1}-\eqref{1.5}, which belongs to $C^{1,1}(\Omega)\cap C^{0,1}(\bar \Omega)$ and satisfies the boundary condition \eqref{1.5} weakly.

We next consider the case when $F$ satisfies F1 and $B>0$. 
Under F1, it is standard that a non-strict subsolution of \eqref{1.1}-\eqref{1.5} can be made strict by using the linearized operator and the mean value theorem.
Consequently, we can have the same uniform solution estimate \eqref{uni sol} as well.
Also, under F1, the uniform global gradient estimate \eqref{uni gra} and uniform interior second derivative estimate \eqref{uni int 2nd} still hold.
From F1 and the interior second derivative estimate \eqref{uni int 2nd}, the operator $F^\epsilon$ satisfies the uniform ellipticity condition in the Evans-Krylov estimates; (see Theorem 17.14 in \cite{GTbook}), and we thus obtain 
\begin{equation}\label{C3}
|u_\epsilon|_{2, \alpha; \Omega^\prime} \le C,
\end{equation}
for some $\alpha \in (0,1)$ and any $\Omega^\prime \subset\subset \Omega$, with constant $C$ independent of $\epsilon$.
Since $F\in C^2(\Gamma)$, $A\in C^2(\bar \Omega\times \mathbb{R}\times \mathbb{R}^n)$, $B\in C^2(\bar \Omega\times \mathbb{R}\times \mathbb{R}^n)$, $\varphi \in C^{1,1}(\partial\Omega\times \mathbb{R})$, $\beta \in C^{1,1}(\partial\Omega)$ and $u_\epsilon \in C^{2,\gamma}(\bar\Omega)$, from the linear Schauder theory in \cite{GTbook} we have $u_\epsilon \in C^{4, \gamma} (\Omega) \cap C^{3,\gamma}(\bar \Omega)$.
Hence, from the uniform estimates \eqref{uni gra estimate} and \eqref{C3}, there exists a subsequence $\{u_{\epsilon_k}\}$ and a function $u\in C^{2,\alpha}(\Omega) \cap C^{0,1}(\bar \Omega)$ such that
\begin{equation}\label{epsilon tends 0}
u_{\epsilon_k} \rightarrow u, \quad {\rm in} \ C^{2,\alpha}(\Omega)\cap C^{0,\alpha^\prime}(\bar \Omega), 
\end{equation}
for $\alpha$ in \eqref{uni int 2nd} and any $\alpha^\prime \in (0, 1)$, as $\epsilon_k \rightarrow 0$. By the stability property of viscosity solutions \cite{CIL1992, E1978, Tru1990}, $u$ is an admissible solution of the problem \eqref{1.1}-\eqref{1.5}, which belongs to $C^{2,\alpha}(\Omega)\cap C^{0,1}(\bar \Omega)$ and satisfies the boundary condition \eqref{1.5} weakly.

In the general case when $F\in C^0(\Gamma)$, we can first approximate $F$ by mollifications, (as in Theorem 17.18 in \cite{GTbook}). Here the cone $\Gamma$ is replaced by a convex set 
$$\Gamma^h = \{r\in \mathbb{S}^n | r+h\xi\in \Gamma, \  \forall \xi\in \mathbb{S}^n \ {\rm satisfying} \  |\xi| = 1\},$$
for small $h>0$, and our approximating mollifications $F^h \in C^2(\Gamma^h)$, satisfy conditions F1$^-$ (or F1), F2 and F3
in $\Gamma^h$ and our previous arguments are applicable.  Then the full strength of Theorem \ref{Th3.1} follows. 
\end{proof}

In Theorem \ref{Th3.1}, we have proved the existence of solutions in $C^{1,1}(\Omega)\cap C^{0,1}(\bar\Omega)$ ($C^{2,\alpha}(\Omega)\cap C^{0,1}(\bar\Omega)$ for some $\alpha\in (0,1)$) of the oblique derivative problem \eqref{1.1}-\eqref{1.5} under F1$^-$ and $B\ge 0$ (F1 and $B>0$). 
Note that Theorem \ref{Th1.1}(i) is just the F1$^-$ and $B\ge 0$ case of Theorem \ref{Th3.1}, we thereby complete the proof of Theorem \ref{Th1.1}(i). 

\begin{remark}\label{Rem 3.2} 
Also from Theorem 1.3 in \cite{JT-oblique-I}, when $B>a_0=-\infty$ in the F1 case of Theorem \ref{Th3.1}, the existence of admissible solutions in $C^{2, \alpha}(\Omega)\cap C^{0,1}(\bar \Omega)$ still holds if $A$ and $B$ satisfy the growth conditions \eqref{quadratic growth} with $B\ge O(1)$. Furthermore Theorem \ref{Th3.1} itself extends to the general situation when $\Gamma$ is replaced by a general convex domain $D$, as in Remark \ref{Rem 1.4}, provided $F$ satisfies F5(0), (F5) in case (a). 
\end{remark}

\subsection*{Alternative hypotheses.}\label{Section 3.1}

We complete this section by elaborating on Remark \ref{Rem 1.1} and in particular discuss alternate hypotheses to uniform regularity of $A$ for gradient estimates. Since the discussion is mainly based on a modification of the proofs of Theorems 1.3 and 3.1 in \cite{JT-oblique-I}, unless otherwise specified, the notation in this subsection follows Section 3 in \cite{JT-oblique-I}.

First we show that the condition F7 can be replaced by F2 in the hypotheses of our gradient estimates in \cite{JT-oblique-I}, at least when $a_0 > -\infty$, provided we strengthen our growth conditions,  (3.31) and (3.33) in \cite{JT-oblique-I},  so that
\begin{equation}\label{3.19}
A= o(|p|^2)I, \quad p\cdot D_pA \le o(|p|^2)I, \quad p\cdot D_pB\le o(|p|^2),
\end{equation}
as $|p| \rightarrow \infty$, uniformly for $x\in\Omega$, $|z|\le M$ for any $M > 0$.  To see this we use the concavity F2 and orthogonal invariance of $\mathcal F$ to imply that if $F(r) = f(\lambda)$, where $\lambda = (\lambda_1, \cdots, \lambda_n)$ denote the eigenvalues of $r \in\Gamma$, then $D_if \le D_jf$ at any fixed point $\lambda$, where $\lambda_i\ge \lambda_j$. 
As pointed out in  \cite{Urbas1995}, this is geometrically evident but it also follows analytically from F2 by applying the mean value theorem to the function $g = D_if - D_jf$ at the points $\lambda$ and $\lambda^*$, where $\lambda^*$ is given by exchanging $\lambda_i$ and $\lambda_j$. We then get $g(\lambda)\le g(\lambda^*) = - g(\lambda)$, by symmetry, and hence $g(\lambda)\le 0$. Returning to our proof of the gradient estimate in Section 3 of \cite{JT-oblique-I} as above, we then obtain the estimate (3.42) in \cite{JT-oblique-I} also for $w_{kk}$ the minimum eigenvalue of $w$, whence $F^{kk} \ge \frac{1}{n} \mathscr{T}$ at the maximum point $x_0$. Without F7 we cannot use the term $KF^{ij}u_iu_j$ in (3.32) in \cite{JT-oblique-I} but this is offset by using our stronger growth conditions \eqref{3.19} and retaining the term  $\mathcal E_2^\prime$ in inequality (3.9) in \cite{JT-oblique-I}. Moreover the details are now technically simpler as we can then replace the function $\eta=e^{K(u-u_0)}$ in the auxiliary function $v$ in (3.13) of \cite{JT-oblique-I} by $\eta = u-u_0$, where $u_0 = \inf_\Omega u$, thereby obtaining in place of inequality (3.32) in \cite{JT-oblique-I},
\begin{equation}\label{3.20}
\begin{array}{rl}
\mathcal{L}\eta \!\!&\!\!\displaystyle =  F^{ij}(w_{ij} + A_{ij}) -(F^{ij}A_{ij}^k + B_{p_k})u_k \\
                \!\!&\!\!\displaystyle \ge F^{ij}w_{ij} - C(1+\mathscr{T})(\omega|Du|^2+1),
\end{array}
\end{equation}
while in place of (3.42) in \cite{JT-oblique-I}, we have simply
\begin{equation}\label{3.26}
w_{kk}  \le  - \frac{1}{6}\alpha M^2_1 + C(\omega|Du|^2+1),
\end{equation}
where $C$ is a constant and $\omega$ a positive decreasing function on $[0,\infty)$ tending to $0$ at infinity, depending on $A,B$ and $M_0 =\sup_\Omega|u|$.
We can then estimate for our fixed $k$, using again our growth condition $A= o(|p|^2)$,
\begin{equation}\label{3.21}
\begin{array}{rl}
\mathcal E_2^\prime \!\!&\!\!\displaystyle =  F^{ij}u_{il}u_{jl} \ge F^{kk}(w_{kk} +A_{kk})^2 \\
                \!\!&\!\!\displaystyle \ge \frac{(\alpha c_0)^2}{2n} M^4 _1\mathscr{T}- C(1+\mathscr{T})(\omega|Du|^2+1)^2.
\end{array}
\end{equation}

Recalling that $F^{ij}w_{ij} \ge 0$, when $a_0 > -\infty$, or more generally when condition F4 in \cite{JT-oblique-I} is satisfied, we can then proceed to recover, in this case, our local and global gradient estimates  in \cite{JT-oblique-I} under these alternative hypotheses to condition F7. Moreover, when $o$ is relaxed to $O$ in \eqref{3.19}, we obtain an estimate in terms of the modulus of continuity of the solution $u$, as in the last assertion of Theorem 3.1 in \cite{JT-oblique-I}, with F7 simply replaced by F2. Consequently we see that Theorem \ref{Th3.1} continues to hold when the condition that $A$ is uniformly regular is replaced by $A$ strictly regular, $|\frac{\beta}{\beta.\nu}- \nu|<1/\sqrt{n}$ and $\mathcal{F}$ orthogonally invariant,  together with any of the conditions (a), (b) or (c) in Remark \ref{Rem 1.1}.

Finally we address the situation when $a_0 = -\infty$ so that we cannot in general bound the term $F^{ij}w_{ij}$ in \eqref{3.20} from below from our other hypotheses. Clearly this can be overcome by a condition such as (3.54) in Theorem 3.1 in \cite{JT-oblique-I} but we will show here first that this condition can be removed altogether from that result. Accordingly we assume
first that $\mathcal F$ is orthogonally invariant satisfying F1, F3 and F7, $|\beta - \nu|<1/\sqrt{n}$ with $A$  and $B$ satisfying the full quadratic structure \eqref{quadratic growth} and $b_0:= \inf\limits_\Omega B(\cdot, u,Du) >a_0$, as in the last assertion of Theorem 3.1 in \cite{JT-oblique-I}.  We now make a further modification of the auxiliary function $v$ in (3.58) in \cite{JT-oblique-I} by taking 
\begin{equation}\label{3.22}
\eta=u-u_0 + K (u-u_0)^2
\end{equation}
where now  $u_0 = \inf\limits_{\Omega \cap B_R}u$ for some ball $B_R = B_R(y)$ of radius $R <1$ and centre $y$ intersecting 
$\partial\Omega$ and $K$ is a positive constant. Then in place of the estimate (3.64) in \cite{JT-oblique-I}, we have
\begin{equation}\label{3.23}
\zeta^2w_{11} \le -\frac{1}{6}\alpha \tilde M^2_1 + C[\zeta^2(|Du|^2+1)  +  \frac{1}{R}\zeta |Du|],
\end{equation}
so that we obtain again $w_{11}(x_0) < 0$ provided $\alpha > C$ is sufficiently large and $\zeta(x_0) |Du(x_0)|> C$, for some constant $C$,  depending on $F, A, B, \Omega, \varphi$ and $M_0$.  With our new choice of $\eta$, we then have from F7
 at the maximum point $x_0$ of $v$ 
\begin{equation}\label{3.24}
\begin{array}{rl}
\mathcal{L}\eta \!\!&\!\!\displaystyle \ge 2KF^{ij}u_iu_j+ [1+ 2K(u-u_0)] [F^{ij}u_{ij} -C(|Du|^2+1)(1+\mathscr T)] \\
                \!\!&\!\!\displaystyle \ge    c_0 K |Du|^2 \mathscr T -(1+2K\theta)[ \sqrt{\mathcal E^\prime_2 \mathscr T}  + C(|Du|^2+1)\mathscr T)]
\end{array}
\end{equation}
for some positive constant $c_0$ depending on $\delta_0$, $\delta_1$ and $n$, provided $\mathop{\rm osc}\limits_{ \Omega \cap B_R} u < \theta$.  Using Cauchy's inequality and choosing $K$ sufficiently large, we then obtain our desired local boundary estimate
\begin{equation}\label{3.25}
|Du(y)| \le \frac{C}{R},
\end{equation}
for  a sufficiently small positive constant $\theta$ depending, along with the constant $C$, on $F, A, B, \Omega, \varphi, \beta$ and $M_0$. Consequently the condition (3.54) is not needed in the hypotheses of Theorem 3.1 in \cite{JT-oblique-I} and from our H\"older estimates in Section 3.3 of \cite{JT-oblique-I} we thus have a global gradient estimate, for any $a_0$, if additionally $\beta = \nu$, $\Gamma\subset \Gamma_k$ with $k>n/2$ and $\Omega$ is convex, which is also applicable to the Dirichlet problem for arbitrary domains $\Omega$, \cite{JT-new}. 

When F7 is replaced by F2 as above, we at least need a control from below, $F^{ij}w_{ij} \ge o(|Du|^4$), at a maximum point of the function $v$, which would follow from a weakening of our condition (3.54) in \cite{JT-oblique-I}, namely
\begin{equation}\label{3.26}
-r\cdot F_r \le  o(|\lambda_0(r)|) \mathscr{T}(r) + O(1) |F(r)|
\end{equation}
as $\lambda_0(r) \rightarrow -\infty$, uniformly for $F(r) >a$ for any constant $a$,  where 
$\lambda_0$ denotes the minimum eigenvalue of $r \in \Gamma$, together with $B\le o(|p|^2)$ in case (b).  Corresponding conditions would also be needed to extend our gradient estimates in cases (a), (b) and (c) for finite $a_0$ above to the more general cases when $\Gamma$ is replaced by a general convex domain $D$, as in Remarks \ref{Rem 1.4} and \ref{Rem 2.4}.

\vspace{3mm}

\section{Comparison principles and uniqueness}\label{Section 4}
\vskip10pt

\rm In this section, we study various comparison principles for weak solutions of the oblique boundary value problem \eqref{1.1}-\eqref{1.5}. In particular we first consider  solutions in $C^{1,1}(\bar\Omega)$ and $C^{2}(\Omega)\cap C^{0,1}(\bar \Omega)$ as a preliminary to the general case of solutions in $C^{1,1}(\Omega)\cap C^{0,1}(\bar \Omega)$. With these results, we complete the proofs of the uniqueness and regularity assertions of Theorem \ref{Th1.1}, as well as the degenerate case in \cite{JT-oblique-I}.

For $F\in C^0(\Gamma)$, the superdifferential of $F$ at $r_0\in \Gamma$ is defined by
\begin{equation}
\partial^+ F(r_0) = \{s\in \mathbb{S}^n| \ F(r) \le F(r_0)+ s\cdot (r-r_0) +o(|r-r_0|) \ {\rm holds} \ {\rm for} \ r \ {\rm near} \ r_0 \}.
\end{equation}
Note that $\partial^+ F(r_0)$ is a closed, convex set, which  may be empty.
When F2 holds, $\partial^+ F(r_0)\neq \emptyset$ for all $r_0\in \Gamma$, and $\partial^+ F(r_0)$ is single valued if $F$ is differentiable at $r_0$ and is multi valued if $F$ is not differentiable at $r_0$.
In this case, we denote 
\begin{equation}\label{supdif Fij}
\{F^{ij}(r_0)\}_{1\le i, j \le n} :=\partial^+F(r_0),
\end{equation}
so that $F^{ij} = \frac{\partial F}{\partial r_{ij}} $ holds almost everywhere in $\Gamma$, and $F^{ij}$ is multi valued in a subset of measure zero in $\Gamma$. For $\{F^{ij}(r_0)\}$ in \eqref{supdif Fij}, we also denote
\begin{equation}\label{trace Fij}
\mathscr{T} = \mathscr{T}(r_0) : = {\rm trace} \{F^{ij}(r_0)\} = [ \mathscr{T^-},\mathscr{T^+}].
\end{equation}
Note that at the points where $F$ is differentiable, $\mathscr{T}$ in \eqref{trace Fij} agrees with $\mathscr{T^-}$ in condition F1$^-$.

We introduce a barrier construction, when $F\in C^0(\Gamma)$ satisfies conditions F1$^-$ and F2, which is a refinement of Part (ii) of Lemma 2.1 in \cite{JT-oblique-II}. 
\begin{Lemma}\label{Lemma 4.1}
Let $u\in C^{1,1}(\Omega)$ be a supersolution of equation \eqref{1.1}, $\underline u\in C^{1,1}(\bar \Omega)$ be a strict subsolution of equation \eqref{1.1} satisfying $\underline u \ge u$ in $\Omega$. Assume that $F$ satisfies F1$^-$ and F2, $A\in C^2(\bar \Omega \times \mathbb{R}\times \mathbb{R}^n)$ is regular and nondecreasing in $z$, $B\ge a_0, \in C^2(\bar \Omega \times \mathbb{R}\times \mathbb{R}^n)$ is convex in $p$ and nondecreasing in $z$. Then for $\eta =e^{K(\underline u-u)}$, the estimate
\begin{equation}\label{barrier inequality}
\mathcal{L} \eta \ge \delta_1 (1+\mathscr{T}) 
\end{equation}
holds almost everywhere in $\Omega$, where $K$ is a sufficiently large positive constant,  $\delta_1$ is a positive constant,  $\mathscr{T}$ is defined in \eqref{trace Fij}, and $\mathcal{L}$ is the operator in \eqref{linearized op full} with $F^{ij}$ defined in \eqref{supdif Fij}.
\end{Lemma}
\begin{proof}
Note first that $\mathcal{L}$ in \eqref{linearized op full} and $\mathscr{T}$ in \eqref{trace Fij} still make sense for $F\in C^0(\Gamma)$ satisfying F1$^-$ and F2.
Since $\underline u$, $\ge u$, is a strict subsolution of equation \eqref{1.1}, from the monotonicity conditions of $A$ and $B$ in $z$, it is readily checked that $\underline u$ satisfies the strict subsolution condition (2.16) in \cite{JT-oblique-II}. Consequently, following the lines of the proof of Lemma 2.1(ii) in \cite{JT-oblique-II}, the inequality \eqref{barrier inequality} holds at the points where $u$ and $\underline u$ are twice differentiable.
\end{proof}

\begin{remark}\label{Rem about subsol}
Note that if $A$ and $B$ are independent of $z$, the assumption $\underline u \ge u$ in $\Omega$ is not needed for the barrier inequality \eqref{barrier inequality} in Lemma \ref{Lemma 4.1}, see \cite{JTY2013, JTY2014} for the Monge-Amp\`ere operator and $k$-Hessian operator cases.
\end{remark}

If the function $\eta$ in \eqref{barrier inequality} is replaced by $\tilde \eta : =\eta - \sup_{\Omega}\eta$, the barrier inequality \eqref{barrier inequality} still holds. Therefore, we can always assume $\eta\le 0$ in $\Omega$.

Based on the barrier inequality \eqref{barrier inequality} in Lemma \ref{Lemma 4.1}, we now present a comparison principle for $C^{1,1}$ solutions in the following theorem.
\begin{Theorem}\label{Th4.1}
Let $u, v \in C^{1,1}(\Omega)$ be a supersolution and a subsolution of equation \eqref{1.1} respectively. Assume that $F$ satisfies conditions F1$^-$-F3 in the cone $\Gamma\subset \mathbb{S}^n$, $A\in C^2(\bar \Omega \times \mathbb{R} \times \mathbb{R}^n)$ is regular and nondecreasing in $z$, $B\ge a_0, \in C^2(\bar \Omega \times \mathbb{R} \times \mathbb{R}^n)$ is convex in $p$ and nondecreasing in $z$. Assume also there exists a strict subsolution $\underline u\in C^{1,1}(\bar \Omega)$ of equation \eqref{1.1} satisfying $\underline u \ge u$ in $\Omega$. Then we have
\begin{equation}\label{reduction to boundary}
\sup_{\Omega} (v-u) \le \sup_{\partial\Omega} (v-u)^+,
\end{equation}
where $(v-u)^+ = \max\{v-u, 0\}$.
\end{Theorem}
\begin{proof}
For $\tau >0$, we suppose that $v-(u-\tau \tilde \eta)$ attains its positive maximum at a point $x_0\in \Omega$, namely 
$$v(x_0)-(u-\tau \tilde\eta)(x_0)=\max_{\bar \Omega}[v-(u-\tau \tilde \eta)]>0,$$
where $\tilde\eta=\eta - \sup_{\Omega}\eta$, and $\eta=e^{K(\underline u-u)}$ is the barrier function in Lemma \ref{Lemma 4.1}.
Since $\tilde\eta\le 0$ in $\Omega$, we can have 
$$v(x_0)>u(x_0).$$
Since the functions $u, v$ and $\tilde\eta$ belong to $C^{1,1}(\Omega)$, from Bony maximum principle in \cite{Bony1967, Lions1983} we have
\begin{equation}\label{3.6}
\mathop{\lim \sup}\limits_{y\rightarrow x_0}\bar \lambda (y) \le 0, \quad [Dv-D(u-\tau \tilde\eta)](x_0) =0,
\end{equation}
where $\bar \lambda(y)$ is the largest eigenvalue of $[D^2v-D^2(u-\tau \tilde\eta)](y)$. Using the definitions of the supersolution $u$ and subsolution $v$, we have
\begin{align}
0 \le &  \mathop{\lim \sup }\limits_{y\rightarrow x_0} \{[\mathcal{F}[v]-B(\cdot, v, Dv)](y) -[\mathcal{F}[u]- B(\cdot, u, Du)](y)\} \nonumber \\ 
   \le &  \mathop{\lim \sup }\limits_{y\rightarrow x_0}\{ F^{ij}(M[u](y)) [ D_{ij}v - D_{ij}u - (A_{ij}(\cdot, v, Dv) -  A_{ij}(\cdot, u, Du)) ](y) \nonumber \\
        & - [ B(\cdot, v, Dv) - B(\cdot, u, Du)](y) \}   \nonumber  \\
   \le &  \mathop{\lim \sup }\limits_{y\rightarrow x_0}\{ F^{ij}(M[u](y)) [ -\tau D_{ij}\tilde \eta - (A_{ij}(\cdot, u, Dv) -  A_{ij}(\cdot, u, Du)) ](y)  \label{3.7} \\
        & - [ B(\cdot, u, Dv) - B(\cdot, u, Du)](y) \}  \nonumber \\
   =   &  \mathop{\lim \sup }\limits_{y\rightarrow x_0}\{ -\tau \mathcal{L}\tilde \eta (y)- F^{ij}(M[u](y)) [ A_{ij}(\cdot, u, Dv) -  A_{ij}(\cdot, u, Du) + \tau D_{p_k}A_{ij}(\cdot, u, Du)D_k\tilde \eta](y)  \nonumber  \\
        & - [ B(\cdot, u, Dv) - B(\cdot, u, Du) + \tau D_{p_k}B(\cdot, u, Du)D_k\tilde \eta](y) \}  \nonumber      
\end{align}
where F2 and \eqref{supdif Fij} are used in the second inequality, the inequality in \eqref{3.6} and the monotonicity of $A$ and $B$ are used in the third inequality.
Using the subadditivity of $\mathop{\lim \sup }\limits_{y\rightarrow x_0}$, we have
\begin{align}
        &  \mathop{\lim \sup }\limits_{y\rightarrow x_0}\{ -\tau \mathcal{L}\tilde \eta (y)- F^{ij}(M[u](y)) [ A_{ij}(\cdot, u, Dv) -  A_{ij}(\cdot, u, Du) + \tau D_{p_k}A_{ij}(\cdot, u, Du)D_k\tilde \eta](y)  \nonumber \\
        & - [ B(\cdot, u, Dv) - B(\cdot, u, Du) + \tau D_{p_k}B(\cdot, u, Du)D_k\tilde \eta](y) \}  \nonumber \\
   \le & -\tau  \mathop{\lim \inf }\limits_{y\rightarrow x_0} \mathcal{L}\tilde \eta (y) - [ B(\cdot, u, Dv) - B(\cdot, u, Du) - D_{p_k}B(\cdot, u, Du)D_k(v-u)](x_0) \nonumber \\
        & - [ A_{ij}(\cdot, u, Dv) -  A_{ij}(\cdot, u, Du) - D_{p_k}A_{ij}(\cdot, u, Du)D_k(v-u)](x_0) [ \mathop{\lim \inf }\limits_{y\rightarrow x_0}F^{ij}(M[u](y)) ] \nonumber \\
   =   &  -\tau  \mathop{\lim \inf }\limits_{y\rightarrow x_0} \mathcal{L}\tilde \eta (y) - [\frac{1}{2} D_{p_kp_l}B(\cdot, u, \bar p) D_k(v-u)D_l(v-u) ](x_0) \label{3.9} \\
        & -[\frac{1}{2} D_{p_kp_l}A_{ij}(\cdot, u, \tilde p) D_k(v-u)D_l(v-u) ](x_0) [ \mathop{\lim \inf }\limits_{y\rightarrow x_0}F^{ij}(M[u](y)) ] \nonumber \\
   \le & [1+ \mathop{\lim \inf }\limits_{y\rightarrow x_0}\mathscr{T}(M[u](y))] (-\tau \delta_1 + C|D(v-u)(x_0)|^2)     \nonumber \\
   \le & [1+\mathop{\lim \inf }\limits_{y\rightarrow x_0}\mathscr{T}(M[u](y))] [-\tau \delta_1 + C\tau^2 \sup_\Omega (|D\tilde\eta|^2)] \nonumber \\
   <   & 0, \nonumber
\end{align}
by taking $\tau$ sufficiently small such that $\tau \in (0, \delta_1 / \{C\sup_{\Omega}(|D\tilde\eta|^2)\})$ if $\sup_{\Omega}(|D\tilde\eta|^2)$ is bounded, where Taylor's formula is used in the equality with $\tilde p=tDu+(1-t)Dv$ and $\bar p=sDu+(1-s)Dv$ for some $t, s\in (0,1)$, and the barrier inequality \eqref{barrier inequality} in Lemma \ref{Lemma 4.1} is used to obtain the second inequality.

Combining \eqref{3.7} with \eqref{3.9}, we get a contradiction.
Then $v-(u-\tau\tilde \eta)$ can only take its positive maximum on $\partial\Omega$, namely
\begin{equation}\label{reduction to boundary tilde}
\sup_{\Omega} [v-(u-\tau \tilde\eta) ] \le \sup_{\partial\Omega} [v-(u-\tau \tilde\eta) ]^+,
\end{equation} 
Letting $\tau \rightarrow 0$, the conclusion \eqref{reduction to boundary} is now proved. 

In the case when $\sup_{\Omega}(|D\tilde\eta|^2)$ in \eqref{3.9} is unbounded, we can repeat the above argument with $\Omega$ replaced by the parallel approximating domains $\Omega_{\epsilon}=\{x\in \Omega| \ {\rm dist}(x, \partial \Omega)>\epsilon\}$ for $\epsilon>0$ sufficiently small. Since $\sup_{\Omega_\epsilon}(|D\tilde\eta|^2)$ is bounded for $\epsilon>0$, \eqref{reduction to boundary tilde} still holds with $\Omega$ replaced by $\Omega_\epsilon$. Then by letting $\epsilon\rightarrow 0$ and $\tau \rightarrow 0$, we also get the conclusion \eqref{reduction to boundary}.
\end{proof}

\begin{remark}
In the proof of Theorem \ref{Th4.1}, the infimum and the supremum in the notation ``$\mathop{\lim \inf }\limits_{y\rightarrow x_0}$'' and ``$\mathop{\lim \sup }\limits_{y\rightarrow x_0}$'' should be understood in the sense of essential infimum and essential supremum, respectively.
We remark that alternatively we can directly use Proposition 1 in \cite{Lions1983} to get the proof of Theorem \ref{Th4.1}, which can avoid such a limiting process.
\end{remark}

\begin{remark}
When only $\varphi$ is strictly increasing in $z$, since $v-u$ can not take positive maximum on $\partial\Omega$, the uniqueness for $C^{1,1}(\bar \Omega)$ solutions in Corollaries 4.1, 4.2 and 4.3 in \cite{JT-oblique-I} follows from the comparison principle \eqref{reduction to boundary} in Theorem \ref{Th4.1}. These uniqueness results have already been foreshadowed at the end of Section 4.3 in \cite{JT-oblique-I}.
\end{remark}

Next, when F1 holds, we consider the comparison principle for the solutions  in the class $C^{2}(\Omega)\cap C^{0,1}(\bar \Omega)$ of the oblique boundary value problem \eqref{1.1}-\eqref{1.5}.
\begin{Theorem}\label{Th4.2}
Let $u, v \in C^{2}(\Omega)\cap C^{0,1}(\bar \Omega)$ be a supersolution and a subsolution of the oblique boundary value problem \eqref{1.1}-\eqref{1.5} respectively, $\Omega\subset \mathbb{R}^n$ with $\partial\Omega\in C^2$. Assume that $F$ satisfies conditions F1-F3 in the cone $\Gamma\subset \mathbb{S}^n$, $B> a_0$ and $\varphi\in C^0(\partial \Omega \times \mathbb{R})$. 
Assume also that $A, B$ and $\varphi$ are nondecreasing in $z$.
Assume further that at least one of the following conditions holds:
\begin{itemize}
\item[(i)] $\varphi$ is strictly increasing in $z$, $A\in C^2(\bar \Omega \times \mathbb{R} \times \mathbb{R}^n)$ is regular, $B\in C^2(\bar \Omega \times \mathbb{R} \times \mathbb{R}^n)$ is convex in $p$, and there exists a strict subsolution $\underline u\in C^{2}(\bar \Omega)$ of equation \eqref{1.1} satisfying $\underline u \ge u$ in $\Omega$;

\item[(ii)]  $A\in C^1(\bar \Omega \times \mathbb{R} \times \mathbb{R}^n)$ is strictly increasing in $z$;

\item[(iii)]  $B\in C^1(\bar \Omega \times \mathbb{R} \times \mathbb{R}^n)$ is strictly increasing in $z$, and $\mathscr{T}(r)$ is bounded from above for $r\in \Gamma$.
\end{itemize}
Then we have
\begin{equation}\label{comparison in Th4.2}
u\ge v, \quad {\rm in} \ \bar \Omega.
\end{equation}
\end{Theorem}
\begin{proof}
In case (i), by Theorem \ref{Th4.1}, the inequality \eqref{reduction to boundary} holds. Consequently, we can assume that $v-u$ attains its positive maximum at a point $z \in \partial\Omega$. Since F1 holds and $u, v \in C^{2}(\Omega)\cap C^{0,1}(\bar \Omega)$ satisfy $\mathcal{G}[u]\le 0$ and $\mathcal{G}[v] \ge 0$ on $\partial\Omega$ weakly, for any admissible functions $\phi, \psi \in C^2(\bar\Omega)$ satisfying $u\le \phi$, $v\ge \psi$ in $\bar \Omega$, $\phi(z)=u(z)< v(z)=\psi(z)$, we have
$$\mathcal{G}[\phi](z)\le 0, \quad {\rm and} \ \  \mathcal{G}[\psi](z) \ge 0.$$ 
Then we have
\begin{align}
0 & \le \mathcal{G}[\psi](z) - \mathcal{G}[\phi](z) \nonumber \\
   & =   \beta(z) \cdot D(\psi-\phi)(z) + [\varphi(z, \phi(z)) - \varphi(z, \psi(z))]  \label{contradiction 1 in Th3.2} \\
   & < 0, \nonumber
\end{align}
since $\beta(z) \cdot D(\psi-\phi)(z)\le 0$ and $\varphi(z, \phi(z)) - \varphi(z, \psi(z))<0$. The contradiction \eqref{contradiction 1 in Th3.2} implies the conclusion \eqref{comparison in Th4.2}.

Now we consider the cases (ii) and (iii).  
We suppose that $\max_{\bar\Omega}(v-u)=:\theta >0$. We may suppose that $\max_{\partial\Omega}(v-u)=:\theta$, otherwise we will get a contradiction. In fact, if there is a point $x_0\in \Omega$ such that $v(x_0)-u(x_0)=\theta$, we have $v(x_0)>u(x_0)$, $Dv(x_0)=Du(x_0)$ and $D^2v(x_0) \le D^2u(x_0)$. Consequently, from the definitions of supersolution $u$ and subsolution $v$, and using F1 and the nondecreasing properties for both $A$ and $B$ as well as the strictly increasing property for either $A$ or $B$, we have
\begin{align}
0 \le & \{\mathcal{F}[v](x_0) - B(x_0, v(x_0), Dv(x_0))\} -\{\mathcal{F}[u](x_0)-B(x_0, u(x_0), Du(x_0))\} \nonumber \\
     = & F(D^2v(x_0)-A(x_0, v(x_0), Dv(x_0)) - F(D^2u(x_0)-A(x_0, u(x_0), Du(x_0)) \label{max F Th3.2} \\
        & + B(x_0, u(x_0), Du(x_0)) - B(x_0, v(x_0), Dv(x_0)) \nonumber  \\
     < & 0,  \nonumber 
\end{align}
which leads to a contradiction.
Then we can assume that $v-u$ attains its positive maximum at a point $z \in \partial\Omega$. 
We consider the function
\begin{equation}\label{Phi x y}
\Phi (x,y) : = v(x) - u(y) - \phi(x,y), 
\end{equation}
with
\begin{equation}\label{phi x y}
\phi(x,y): = \frac{1}{2\epsilon} a^{ij}(z) (x_i-y_i)(x_j-y_j) + \varphi(z, v(z)) \beta(z)\cdot (x-y) - \delta (d(x)+ d(y)) + \delta |x-z|^2,
\end{equation}
where $\epsilon, \delta$ are positive constants, $d\in C^2(\bar\Omega)$ is a positive defining function for $\Omega$ which agrees with the distance function $d(\cdot)={\rm dist} (\cdot, \partial\Omega)$ in a neighbourhood of  $\partial\Omega$, the symmetric matrix $a(\cdot)=\{a^{ij}(\cdot)\}\in C^2(\bar \Omega)$ satisfies
\begin{equation}\label{relation between beta nu}
a(\cdot) \ge a^0 I,  \ {\rm in} \ \ \bar\Omega, \quad\quad a^{ij}(\cdot) \beta_j(\cdot) =\nu_i(\cdot),  \ \  {\rm on} \ \partial\Omega,
\end{equation}
for some constant $a^0>0$, and $i=1,\cdots, n$. Note that \eqref{relation between beta nu} is guaranteed by the obliqueness $\beta\cdot\nu >0$ on $\partial\Omega$.
Clearly, when $\beta\equiv \nu$, we can just take $a(\cdot)=I$ in $\bar \Omega$.
For $\epsilon>0$, let $(x_\epsilon, y_\epsilon)$ be a maximum point of $\Phi(x,y)$. Since $v(x)-u(x)- \delta |x-z|^2$ has $x=z$ as a unique maximum point for $\delta>0$, it is standard to obtain
\begin{equation}\label{convergence}
x_\epsilon \rightarrow z, \quad \frac{1}{\epsilon} a^{ij}(z) ({x_\epsilon}-{y_\epsilon})_i({x_\epsilon}-{y_\epsilon})_j \rightarrow 0, \quad v(x_\epsilon)\rightarrow v(z), \quad  u(y_\epsilon)\rightarrow u(z),
\end{equation}
as $\epsilon\rightarrow 0$.
For simplicity, we write $(\hat x, \hat y)$ for $(x_\epsilon, y_\epsilon)$. 
Since \eqref{convergence} holds, by taking sufficiently small $\epsilon$ and $\delta$, we can have
\begin{equation}\label{v>u}
v(\hat x) - u(\hat y) \ge  \frac{\theta}{2} >0.
\end{equation}
Since $\Omega$ is bounded and $\partial\Omega\in C^2$, then $\Omega$ satisfies the uniform exterior sphere condition, namely there exists $r>0$, such that 
\begin{equation}\label{UESC}
B(z-r\nu(z), r)\cap \Omega = \emptyset,
\end{equation}
for $z\in \partial\Omega$, where $B(z-r\nu(z), r)$ denotes the closed ball of radius $r$ centered at $z-r\nu(z)$, $\nu(z)$ is the unit inner normal vector at $z$. For $\Omega$ satisfying \eqref{UESC}, since $|x-z+r\nu(z)|\ge r$ for $z\in \partial\Omega$ and $x\in \bar \Omega$, we have
\begin{equation}\label{EUESC}
\nu(z) \cdot (z-x) \le \frac{|z-x|^2}{2r},
\end{equation} 
for $z\in \partial\Omega$ and $x\in \bar \Omega$. The geometric property \eqref{EUESC} was first observed in \cite{L1983}, and later used in \cite{LS1984, L1985, IL1990} and etc.

If $\hat x \in \partial\Omega$, taking $\xi(x)=u(y)+\phi(x,y)$, we see that $v(x)-\xi(x)$ attains its maximum at $\hat x \in \partial\Omega$. Hence, from the definition of viscosity subsolution, we have
\begin{equation}\label{xi vis sub}
\beta(\hat x)\cdot D_x\xi(\hat x) - \varphi(\hat x, v(\hat x)) \ge 0. 
\end{equation}
By calculations and using $Dd = \nu$ on $\partial\Omega$, we have
\begin{align}
                                      &  \beta(\hat x)\cdot D_x\xi(\hat x) - \varphi(\hat x, v(\hat x)) \nonumber \\
                                   = & \frac{1}{\epsilon} a^{ij}(z)\beta_i(\hat x)(\hat x_j - \hat y_j) + \varphi(z,v(z))\beta(\hat x)\cdot \beta(z) \label{G xi x} \\
                                      & - \delta \beta(\hat x)\cdot \nu(\hat x) + 2 \beta(\hat x)\cdot (\hat x-z) - \varphi(\hat x, v(\hat x)).  \nonumber
\end{align}
By using \eqref{relation between beta nu} and \eqref{EUESC}, we can estimate the first term on the right hand side of \eqref{G xi x}, namely
\begin{align}
\frac{1}{\epsilon} a^{ij}(z)\beta_i(\hat x)(\hat x_j - \hat y_j) =  &  \frac{1}{\epsilon} a^{ij}(\hat x)\beta_i(\hat x)(\hat x_j - \hat y_j) +  \frac{1}{\epsilon} [a^{ij}(z)-a^{ij}(\hat x)]\beta_i(\hat x)(\hat x_j - \hat y_j) \label{tricky term} \\
                                                                                          \le & \frac{|\hat x - \hat y|^2}{2\epsilon r} + \frac{C}{\epsilon} |\hat x-z||\hat x - \hat y|, \nonumber
\end{align}
for some positive constant $C$ depending on $\|a^{ij}\|_{C^1(\bar \Omega)}$. Combining \eqref{G xi x} with \eqref{tricky term}, and using the obliqueness \eqref{ob} and the convergence \eqref{convergence}, we get
\begin{align}
     & \beta(\hat x)\cdot D_x\xi(\hat x) - \varphi(\hat x, v(\hat x)) \nonumber \\
                                   \le & \frac{|\hat x - \hat y|^2}{2\epsilon r} + \frac{C}{\epsilon} |\hat x-z||\hat x - \hat y| + \varphi(z,v(z))\beta(\hat x)\cdot \beta(z) \label{G xi x 1} \\
                                        & - \delta \beta_0 + 2 |\hat x - z| - \varphi(\hat x, v(\hat x)) \nonumber \\
                                   \le & -\frac{1}{2}\delta \beta_0 <0, \nonumber
\end{align}
for sufficiently small $\epsilon>0$, which leads to a contradiction with \eqref{xi vis sub}. 

If $\hat y \in \partial\Omega$, taking $\eta (y)= v(x)-\phi(x,y)$, we see that $u(y)-\eta(y)$ attains its minimum at $\hat y \in \partial\Omega$. Hence, from the definition of viscosity supersolution, we have 
\begin{equation}\label{eta vis super}
\beta(\hat y)\cdot D_y\eta(\hat y) - \varphi(\hat y, u(\hat y))\le 0. 
\end{equation}
By similar calculations as in \eqref{G xi x}, \eqref{tricky term} and \eqref{G xi x 1}, and using \eqref{v>u} and the monotonicity of $\varphi$, we have
\begin{align}
       & \beta(\hat y)\cdot D_y\eta(\hat y) - \varphi(\hat y, u(\hat y)) \nonumber \\
  =   & -\frac{1}{\epsilon}a^{ij}(z)\beta_i(\hat y)(\hat y_j - \hat x_j) + \varphi(z, v(z))\beta (\hat y)\cdot \beta(z) + \delta \beta(\hat y)\cdot \nu(\hat y) - \varphi(\hat y, u(\hat y)) \label{G xi x 2} \\
\ge  & - \frac{|\hat y-\hat x|^2}{2\epsilon r} - \frac{C}{\epsilon}|\hat y-z||\hat y-\hat x|  + \varphi(z, v(z))\beta (\hat y)\cdot \beta(z) + \delta \beta_0 - \varphi(\hat y, v(\hat x)) \nonumber \\
\ge  & \frac{1}{2} \delta \beta_0 >0, \nonumber
\end{align}
for sufficiently small $\epsilon>0$, which leads to a contradiction with \eqref{eta vis super}. 

We have now proved that for small enough $\epsilon>0$, the function $\Phi(x,y)$ in \eqref{Phi x y} does not attain a maximum over $\bar \Omega \times \bar \Omega$ on $\partial (\Omega \times \Omega)$.
Then we can assume $\Phi(x,y)$ attains its maximum at $(\hat x, \hat y) \in \Omega\times \Omega$.
At the maximum point $(\hat x, \hat y)$, we have
\begin{equation}\label{max prin Phi}
D_x \Phi = D_y\Phi =0, \quad {\rm and} \ \ D^2_{x,y}\Phi \le 0,
\end{equation}
which leads to
\begin{align}
Dv(\hat x) = & \frac{1}{\epsilon}a(z)(\hat x-\hat y) + \varphi(z,v(z))\beta(z) - \delta Dd(\hat x) + 2\delta (\hat x-z), \label{key relation 1 in Th3.2} \\
Du(\hat y) = & \frac{1}{\epsilon}a(z)(\hat x-\hat y) + \varphi(z,v(z))\beta(z) + \delta Dd(\hat y), \label{key relation 2 in Th3.2}
\end{align}
and 
\begin{equation}\label{key relation 3 in Th3.2}
D^2v(\hat x)  - D^2 u(\hat y)   \le \delta [- D^2d(\hat x)  -  D^2d(\hat y) +2I].
\end{equation}
The inequality \eqref{key relation 3 in Th3.2} follows since, by \eqref{max prin Phi}, $(\xi, \xi)(D^2_{x,y}\Phi)(\xi, \xi)^T \le 0$ for any $\xi \neq 0, \in \mathbb{R}^n$. 
Then we have
\begin{align}
0   \le  & [\mathcal{F}[v]-B(\cdot, v, Dv)](\hat x) - [\mathcal{F}[u]-B(\cdot, u, Du)](\hat y) \nonumber \\
     \le  & F^{ij}(M[u](\hat y))\{[v_{ij}(\hat x)-A_{ij}(\hat x, v(\hat x), Dv(\hat x))]-[u_{ij}(\hat y)-A_{ij}(\hat y, u(\hat y), Du(\hat y))]\} \nonumber \\
           & + B(\hat y, u(\hat y), Du(\hat y)) - B(\hat x, v(\hat x), Dv(\hat x)) \nonumber \\
     \le  & F^{ij}(M[u](\hat y))\{A_{ij}(\hat y, u(\hat y), Du(\hat y))-A_{ij}(\hat x, v(\hat x), Dv(\hat x)) \label{Theta function} \\
           & + \delta [-D_{ij}d(\hat x) - D_{ij}d(\hat y) + 2\delta_{ij}]\} + B(\hat y, u(\hat y), Du(\hat y)) - B(\hat x, v(\hat x), Dv(\hat x)) \nonumber \\      
      =:  & \Theta(\hat x, \hat y),   \nonumber
\end{align}
where F2 is used to obtain the second inequality, and \eqref{key relation 3 in Th3.2} is used to obtain the third inequality. By using the mean value theorem, we have
\begin{align} 
   & A_{ij}(\hat y, u(\hat y), Du(\hat y))-A_{ij}(\hat x, v(\hat x), Dv(\hat x)) \nonumber \\
 =& [A_{ij}(\hat y, u(\hat y), Du(\hat y))-A_{ij}(\hat x, u(\hat x), Du(\hat x))] + [A_{ij}(\hat x, u(\hat x), Du(\hat x)) \nonumber \\
   & - A_{ij}(\hat x, v(\hat x), Du(\hat x))]  + [A_{ij}(\hat x, v(\hat x), Du(\hat x)) - A_{ij}(\hat x, v(\hat x), Dv(\hat x))] \nonumber \\
 =& D_xA_{ij}(\bar x, u(\hat y), Du(\hat y)) (\hat y-\hat x)  + D_zA_{ij}(\hat x, \bar z, Du(\hat y)) (u(\hat y)-v(\hat x)) \label{A MVT} \\
   & + \delta D_pA_{ij}(\hat x, v(\hat x), \bar p) [Dd (\hat x)+Dd (\hat y) - 2(\hat x-z)], \nonumber
\end{align}
where $\bar x = t_1\hat y+(1-t_1)\hat x$, $\bar z = t_2 u(\hat y)+(1-t_2)v(\hat x)$ and $\bar p=t_3 Du(\hat y)+(1-t_3)Dv(\hat x)$ for some $t_1, t_2, t_3 \in (0,1)$, \eqref{key relation 1 in Th3.2} and \eqref{key relation 2 in Th3.2} are used in the last equality.
Similarly, we get
\begin{align} 
   & B(\hat y, u(\hat y), Du(\hat y))-B(\hat x, v(\hat x), Dv(\hat x)) \nonumber \\
 =& D_xB(\tilde x, u(\hat y), Du(\hat y)) (\hat y-\hat x)  + D_zB(\hat x, \tilde z, Du(\hat y)) (u(\hat y)-v(\hat x)) \label{B MVT} \\
   & + \delta D_pB(\hat x, v(\hat x), \tilde p) [Dd (\hat x)+ Dd (\hat y) - 2(\hat x-z)], \nonumber
\end{align}
where $\tilde x = s_1\hat y+(1-s_1)\hat x$, $\tilde z = s_2 u(\hat y)+(1-s_2)v(\hat x)$, and $\tilde p=s_3 Du(\hat y)+(1-s_3)Dv(\hat x)$, for some $s_1, s_2, s_3 \in (0,1)$.
By \eqref{A MVT} and \eqref{B MVT}, we have
\begin{align}
\Theta(\hat x, \hat y) = & F^{ij}(M[u](\hat y)) [D_zA_{ij}(\hat x, \bar z, Du(\hat y)) (u(\hat y)-v(\hat x)) + R^A_{ij}(\hat x, \hat y)] \label{Theta function 1} \\
                     & + D_zB(\hat x, \tilde z, Du(\hat y)) (u(\hat y)-v(\hat x)) + R^B(\hat x, \hat y),  \nonumber
\end{align}
where
\begin{align}
R^A_{ij}(\hat x, \hat y): = & D_xA_{ij}(\bar x, u(\hat y), Du(\hat y)) (\hat y-\hat x) + \delta D_pA_{ij}(\hat x, v(\hat x), \bar p) [Dd (\hat x)+Dd (\hat y) - 2(\hat x-z)] \label{RA} \\
                     & + \delta [-D_{ij}d(\hat x) - D_{ij}d(\hat y) + 2\delta_{ij}], \nonumber \\
R^B(\hat x, \hat y)      :=  & D_xB(\tilde x, u(\hat y), Du(\hat y)) (\hat y-\hat x) + \delta D_pB(\hat x, v(\hat x), \tilde p) [Dd (\hat x)+Dd (\hat y) - 2(\hat x-z)]. \label{RB} 
\end{align}

In case (ii), we have 
\begin{equation}\label{strictly increasing A}
D_zA(x,z,p) \ge c_0 I, 
\end{equation}
for all $(x,z,p)\in \Omega \times \mathbb{R} \times \mathbb{R}^n$ and some constant $c_0>0$. 
Since condition F5 holds when conditions F1, F2 and F3 hold, we have $\mathscr{T}(r)\ge \delta_0$ for some positive constant $\delta_0$.
By \eqref{v>u} and \eqref{strictly increasing A}, since $B$ is nondecreasing in $z$, we have
\begin{align}
\Theta(\hat x, \hat y) \le & \big [ -\frac{1}{2}\theta c_0 + \sum\limits_{i,j=1}^n |R^A_{ij}(\hat x, \hat y)| \big ] \mathscr{T}(M[u](\hat y)) + R^B(\hat x, \hat y) \nonumber  \\
                                 \le & -\frac{1}{4}\theta c_0 \delta_0 +  |R^B(\hat x, \hat y)| \label{Theta function 2}\\
                                 \le & -\frac{1}{8}\theta c_0 \delta_0 < 0, \nonumber
\end{align}
where F5 is used in the second inequality, $\epsilon$ and $\delta$ are chosen sufficiently small such that $\sum\limits_{i,j=1}^n |R^A_{ij}(\hat x, \hat y)|$ and $|R^B(\hat x, \hat y)|$ can be made as small as we want.
From \eqref{Theta function} and \eqref{Theta function 2}, we get a contradiction, which implies the conclusion \eqref{comparison in Th4.2}.

In case (iii), we have 
\begin{equation}\label{strictly increasing B}
D_zB(x,z,p) \ge c_1, 
\end{equation}
for all $(x,z,p)\in \Omega \times \mathbb{R} \times \mathbb{R}^n$ and some constant $c_1>0$. For $r\in \Gamma$, we have
\begin{equation}\label{bounded trace}
\mathscr{T}(r)\le T,
\end{equation}
for some positive constant $T$.
By \eqref{v>u} and \eqref{strictly increasing B}, since $A$ is nondecreasing in $z$, we have
\begin{align}
\Theta(\hat x, \hat y) \le &  -\frac{1}{2}\theta c_1 + \sum\limits_{i,j=1}^n |R^A_{ij}(\hat x, \hat y)| \mathscr{T}(M[u](\hat y)) + R^B(\hat x, \hat y) \nonumber \\
                                 \le & -\frac{1}{2}\theta c_1 + T \sum\limits_{i,j=1}^n |R^A_{ij}(\hat x, \hat y)|  + |R^B(\hat x, \hat y)|  \label{Theta function 3} \\
                                 \le & -\frac{1}{4}\theta c_1 < 0, \nonumber
\end{align}
where \eqref{bounded trace} is used in the second inequality, $\epsilon$ and $\delta$ are chosen sufficiently small such that $\sum\limits_{i,j=1}^n |R^A_{ij}(\hat x, \hat y)|$ and $|R^B(\hat x, \hat y)|$ can be made as small as we want. From \eqref{Theta function} and \eqref{Theta function 3}, we get a contradiction, which implies the conclusion \eqref{comparison in Th4.2}.
\end{proof}

\begin{remark}
Theorem \ref{Th4.2}(i) depends on the barrier construction in Lemma \ref{Lemma 4.1} and the comparison principle in Theorem \ref{Th4.1}. The functions $\Phi(x,y)$ in \eqref{Phi x y} and $\phi(x,y)$ in \eqref{phi x y} play important roles in the proof of Theorem \ref{Th4.2}(ii)(iii) and are particularly constructed for the oblique boundary value condition by modifications of the analogous functions in \cite{CIL1992, IL1990, K2004}. 
\end{remark}

\begin{remark}
The condition that $\mathscr{T}(r)$ is bounded from above  is satisfied by the Hessian quotient operators $\mathcal F_{k,k-1}$, given by $F_{k,k-1}=\frac{S_k}{S_{k-1}}$,
 and by our degenerate operators $\mathfrak{M}_k$,  introduced in Section 4.3 of \cite{JT-oblique-I} in the respective cones $\Gamma_k$ and $\mathcal P_k$, for $1\le k \le n$.  
\end{remark}

\begin{remark}
In the definition of viscosity subsolution $v$ satisfying \eqref{xi vis sub}, the graphs of $v$ and $\xi$ do not touch at $\hat x \in \partial\Omega$. By modifying the test function by $\tilde \xi(x): = \xi (x) + (v(\hat x)-\xi(\hat x))$, the graphs of $v$ and $\tilde \xi$ can touch at $\hat x$, which corresponds to the viscosity subsolution definition in the introduction. Similarly, in the definition of viscosity supersolution $u(y)$ satisfying \eqref{eta vis super}, the graphs of $u$ and $\eta$ do not touch at $\hat y \in \partial\Omega$. Modifying the test function by $\tilde \eta(y): = \eta (y) + (u(\hat y)-\eta(\hat y))$, $u$ and $\tilde \eta$ can touch at $\hat y$, which corresponds to the viscosity supersolution definition in the introduction. 
\end{remark}

 To complete our comparison principles we now extend our solution space to be the union of the previous two cases, namely
we consider solutions in $C^{1,1}(\Omega)\cap C^{0,1}(\bar \Omega)$.   For this, we need an equivalent definition of viscosity solutions using semi-jets. Accordingly we define the semi-jets $J_K^{2,\pm}u(x)$ of $u: K\rightarrow \mathbb{R}$ at $x\in K$, and their closures $\bar J_K^{2,\pm}u(x)$ by:
\begin{align}
J_K^{2,+}u(x) := &  \{(p,X)\in \mathbb{R}^n \times \mathbb{S}^n| \   u(y)\le u(x) + p\cdot (y-x) \nonumber \\
 & + \frac{1}{2}X(y-x)\cdot (y-x) + o(|y-x|^2) \ {\rm as} \ y\in K\rightarrow x\},  \nonumber
\end{align}
\begin{align}
J_K^{2,-}u(x) := &  \{(p,X)\in \mathbb{R}^n \times \mathbb{S}^n| \   u(y)\ge u(x) + p\cdot (y-x) \nonumber \\
 & + \frac{1}{2}X(y-x) \cdot (y-x) + o(|y-x|^2) \ {\rm as} \ y\in K\rightarrow x\},  \nonumber
\end{align}
and
\begin{align}
\bar J_K^{2,\pm}u(x) := &  \{(p,X)\in \mathbb{R}^n \times \mathbb{S}^n| \  \exists x_k\in K \ {\rm and} \ (p_k, X_k)\in J_K^{2,\pm}u(x) \ {\rm such \ that} \nonumber \\
 &  (x_k, u(x_k), p_k, X_k)\rightarrow (x, u(x), p, X) \ {\rm as} \ k\rightarrow +\infty \},  \nonumber
\end{align}
where $K\subset \mathbb{R}^n$ is a domain, (which is not necessarily open). When $x\in \Omega$, it is obvious that $J_\Omega^{2,\pm}u(x) = J_{\bar \Omega}^{2,\pm}u(x)$, and $\bar J_\Omega^{2,\pm}u(x) = \bar J_{\bar \Omega}^{2,\pm}u(x)$. For convenience, for $x\in \Omega$ we simply write $J^{2,\pm}u(x)$ ($\bar J^{2,\pm}u(x)$) for $J_\Omega^{2,\pm}u(x)=J_{\bar\Omega}^{2,\pm}u(x)$ ($\bar J_\Omega^{2,\pm}u(x)=\bar J_{\bar \Omega}^{2,\pm}u(x)$).

Let $\Omega\subset \mathbb{R}^n$, $u\in C^0(\bar \Omega)$ is a viscosity  subsolution (supersolution) of the boundary value problem \eqref{1.1}-\eqref{1.2} if
\begin{equation}\label{interior vis def}
F(X-A(x, u(x), p)) \ge (\le) B(x, u(x), p),
\end{equation}
holds for $x\in \Omega$, $(p, X)\in \bar J_{\bar \Omega}^{2,+}u(x)$ ($\bar J_{\bar \Omega}^{2,-}u(x)$), and
\begin{equation}\label{boundary vis def}
G(x, u(x), p) \ge (\le) 0, \quad {\rm or} \ \ F(X-A(x, u(x), p)) \ge (\le) B(x, u(x), p),
\end{equation}
holds for $x\in \partial\Omega$, $(p, X)\in \bar J_{\bar \Omega}^{2,+}u(x)$ ($\bar J_{\bar \Omega}^{2,-}u(x)$ and $X-A(x,u(x), p)\in \Gamma$). Then $u\in C^0(\bar \Omega)$ is a viscosity solution of the problem \eqref{1.1}-\eqref{1.2} if it is both a viscosity subsolution and a viscosity supersolution of the problem \eqref{1.1}-\eqref{1.2}. If $\mathcal{F}$ satisfies F1, for $x\in \partial\Omega$, we only need the first inequality in \eqref{boundary vis def}. While if $\mathcal{F}$ satisfies F1$^-$, for $x\in \partial\Omega$, both the inequalities in \eqref{boundary vis def} are needed.

If $F$ only satisfies F1$^-$, the comparison principle in Theorem \ref{Th4.2} also holds for solutions in $C^{1,1}(\Omega)\cap C^{0,1}(\bar \Omega)$. We formulate it in the following theorem which embraces both Theorems \ref{Th4.1} and \ref{Th4.2} as special cases.
\begin{Theorem}\label{Th4.3}
Let $u, v \in C^{1,1}(\Omega)\cap C^{0,1}(\bar \Omega)$ be a supersolution and a subsolution of the oblique boundary value problem \eqref{1.1}-\eqref{1.5} respectively, $\Omega\subset \mathbb{R}^n$ with $\partial\Omega\in C^2$. Assume that $F$ satisfies conditions F1$^-$, F2 and F3 in the cone $\Gamma\subset \mathbb{S}^n$, $B\ge a_0$ and $\varphi\in C^0(\partial \Omega \times \mathbb{R})$. 
Assume also that $A, B$ and $\varphi$ are nondecreasing in $z$.
Assume further that either {\rm (i)}  or {\rm (ii)} or {\rm (iii)} in Theorem \ref{Th4.2} holds. Then the comparison assertion \eqref{comparison in Th4.2} holds.
\end{Theorem}
\begin{proof}
We now use the viscosity notions \eqref{interior vis def} and \eqref{boundary vis def}. We first consider the case (i). Let
\begin{equation}\label{tilde u}
\tilde u:= u-\tau \tilde \eta,
\end{equation} 
where $\tau>0$ is a constant, $\tilde\eta=\eta - \sup_{\Omega}\eta$, and $\eta=e^{K(\underline u-u)}$ is the barrier function in Lemma \ref{Lemma 4.1}. From \eqref{reduction to boundary tilde} in Theorem \ref{Th4.1}, we have
\begin{equation}
\sup_{\Omega}(v-\tilde u) \le \sup_{\partial\Omega} (v-\tilde u)^+,
\end{equation}
for sufficiently small $\tau >0$. Consequently, we can assume that $v-\tilde u$ attains its positive maximum at $z\in \partial\Omega$, namely $v(z)-\tilde u(z)=\max_{\bar\Omega}(v-\tilde u):=\theta >0$.
We consider the function
\begin{equation}\label{tilde Phi x y}
\tilde \Phi (x,y) : = v(x) - \tilde u(y) - \phi(x,y), 
\end{equation}
where $\phi(x,y)$ is the function defined in \eqref{phi x y}.
For $\epsilon>0$, let $(x_\epsilon, y_\epsilon)$ be a maximum point of $\tilde \Phi(x,y)$. Since $v(x)-\tilde u(x)- \delta |x-z|^2$ has $x=z$ as a unique maximum point for $\delta>0$, we have 
\begin{equation}\label{convergence in Th3.3}
x_\epsilon \rightarrow z, \quad \frac{1}{\epsilon} a^{ij}(z) ({x_\epsilon}-{y_\epsilon})_i({x_\epsilon}-{y_\epsilon})_j \rightarrow 0, \quad v(x_\epsilon)\rightarrow v(z), \quad  \tilde u(y_\epsilon)\rightarrow \tilde u(z),
\end{equation}
as $\epsilon\rightarrow 0$.
For simplicity, we write $(\hat x, \hat y)$ for $(x_\epsilon, y_\epsilon)$. 
Since \eqref{convergence in Th3.3} holds, by taking sufficiently small $\epsilon$ and $\delta$, we have
\begin{equation}\label{v>u in Th3.3}
v(\hat x) - u(\hat y) \ge v(\hat x) - \tilde u(\hat y) \ge \frac{\theta}{2} >0.
\end{equation}
At the point $(\hat x, \hat y)$,
by Lemma 3.6 (Ishii's Lemma) and Proposition 2.7 in \cite{K2004}, there exists $X, Y\in \mathbb{S}^n$ such that
\begin{align}
&(p^v, X^v):=(D_x\phi(\hat x, \hat y), X-\delta D^2d(\hat x) + 2\delta I) \in \bar J^{2,+}_{\bar \Omega}v(\hat x), \label{semi-jets v} \\
&(p^{\tilde u}, Y^{\tilde u}):=(- D_y\phi(\hat x, \hat y), Y+\delta D^2d(\hat y)) \in \bar J^{2,-}_{\bar \Omega}\tilde u(\hat y), \label{semi-jets tilde u}
\end{align}
and
\begin{equation}\label{a a a a}
\left (\begin{array}{cc} X & 0 \\ 0 & -Y  \end{array}\right )
\le
\frac{1}{\epsilon} \left (\begin{array}{cc} a(z)+[a(z)]^2 & -a(z)-[a(z)]^2 \\ -a(z)-[a(z)]^2 & a(z)+[a(z)]^2  \end{array}\right ),
\end{equation}
where the matrix $a(z)$ is defined in \eqref{relation between beta nu}.
The form of the matrix on the right hand side of \eqref{a a a a} is obtained by taking $\mu=2/\epsilon$ in Lemma 3.6 in \cite{K2004}.
Note that \eqref{a a a a} implies $X\le Y$. 
Since $\tilde u = u-\tau [e^{K(\underline u-u)} -\sup_\Omega e^{K(\underline u-u)}]$,  we can calculate from \eqref{semi-jets tilde u} that
\begin{equation}
(p^u, Y^u) \in \bar J^{2,-}_{\bar \Omega} u(\hat y),
\end{equation}
where
\begin{align}
p^u = &
\frac{1}{1+\tau K \eta(\hat y)} [-D_y\phi(\hat x, \hat y)+\tau K\eta(\hat y)D\underline u(\hat y)], \label{pu}\\
Y^u = &
\frac{1}{1+\tau K \eta(\hat y)} [Y+ \delta D^2d(\hat y) + \tau K \eta(\hat y) D^2\underline u(\hat y)) \label{Yu}\\
& + \frac{\tau K^2\eta(\hat y)}{(1+\tau K\eta(\hat y))^2} (D\underline u(\hat y)+D_y\phi(\hat x, \hat y))\otimes (D\underline u(\hat y)+D_y\phi(\hat x, \hat y))]. \nonumber
\end{align}
Thus, if $\hat x\in \partial\Omega$, by the definition of viscosity subsolution $v$ under F1$^-$, we have
\begin{equation}\label{v boundary inequality}
\beta(\hat x)\cdot p^v - \varphi(\hat x, v(\hat x)) \ge 0,
\end{equation}
or
\begin{equation}\label{v X inequ}
F(X^v-A(\hat x, v(\hat x), p^v)) \ge B(\hat x, v(\hat x), p^v).
\end{equation}
Observing that $D_x\xi(\hat x)$ in \eqref{G xi x 1} is equal to $p^v$ in \eqref{v boundary inequality}, we get a contradiction from \eqref{G xi x 1} and \eqref{v boundary inequality}.
Therefore, the only possible case is \eqref{v X inequ}. Note that if $\hat x\in \Omega$, the inequality \eqref{v X inequ} holds directly from the definition of the viscosity subsolution $v$.
Similarly, if $\hat y\in \partial\Omega$, by the definition of viscosity subsolution $u$ under F1$^-$, we have
\begin{equation}\label{u boundary inequality pu}
\beta(\hat y)\cdot p^u - \varphi(\hat y, u(\hat y)) \le 0,
\end{equation}
or
\begin{equation}\label{u Yu inequ}
F(Y^u-A(\hat y, u(\hat y), p^u)) \le B(\hat y, u(\hat y), p^u).
\end{equation}
Plugging \eqref{pu} into \eqref{u boundary inequality pu}, we have
\begin{equation}\label{u boundary inequality pu 1}
\beta(\hat y)\cdot p^u - \varphi(\hat y, u(\hat y)) \le \tau K \eta(\hat y) |\beta (\hat y) \cdot D\underline u(\hat y) - \varphi(\hat y, u(\hat y))|.
\end{equation}
Since $D_y\eta(\hat y)$ in \eqref{G xi x 2} is equal to $p^u$ in \eqref{u boundary inequality pu 1}, by choosing $\tau$ sufficiently small, we also get a contradiction from \eqref{G xi x 2} and \eqref{u boundary inequality pu 1}. Therefore, the only possible case is \eqref{u Yu inequ}. Note that if $\hat y\in \Omega$, the inequality \eqref{u Yu inequ} holds directly from the definition of the viscosity supersolution $u$.
Using \eqref{v X inequ} and \eqref{u Yu inequ}, we have
\begin{align}
0 \le & [F(X^v-A(\hat x, v(\hat x), p^v)) - B(\hat x, v(\hat x), p^v)] \nonumber \\
        & - [F(Y^u-A(\hat y, u(\hat y), p^u)) - B(\hat y, u(\hat y), p^u)] \nonumber \\
   \le & F^{ij} [X^v_{ij} - Y^u_{ij}+A_{ij}(\hat y, u(\hat y), p^u)-A_{ij}(\hat x, v(\hat x), p^v)]  \nonumber \\
        & +  B(\hat y, u(\hat y), p^u) - B(\hat x, v(\hat x), p^v) \nonumber \\
   \le & F^{ij} (X^v_{ij} - Y^u_{ij}) - F^{ij}[ A_{ij}(\hat x, u(\hat y), p^v) - A_{ij}(\hat y, u(\hat y), p^u)]  \label{0 le Theta} \\
        & - [ B(\hat x, u(\hat y), p^v) - B(\hat y, u(\hat y), p^u) ] \nonumber \\  
   =:  & \Theta(\hat x, \hat y),   \nonumber      
\end{align}
where F2 is used in the second inequality with $F^{ij}:=F^{ij}(Y^u-A(\hat y, u(\hat y), p^u))$, the monotonicity of $A, B$ and \eqref{v>u in Th3.3} are used in the third inequality. By further calculations, we have
\begin{align}
    \Theta(\hat x, \hat y)=  & F^{ij} (X^v_{ij} - Y^u_{ij}) - [F^{ij}D_{p_k}A_{ij}(\hat y, u(\hat y), p^u) + D_{p_k}B(\hat y, u(\hat y), p^u)] (p^v_k - p^u_k) \nonumber \\ 
        & - F^{ij} [A_{ij}(\hat y, u(\hat y), p^v)-A_{ij}(\hat y, u(\hat y), p^u)-D_{p_k}A_{ij}(\hat y, u(\hat y), p^u)(p^v_k - p^u_k)] \nonumber \\
        & -[B(\hat y, u(\hat y), p^v)-B(\hat y, u(\hat y), p^u)-D_{p_k}B(\hat y, u(\hat y), p^u)(p^v_k - p^u_k)] \nonumber \\
        & - F^{ij}[A_{ij}(\hat x, u(\hat y), p^v) - A_{ij}(\hat y, u(\hat y), p^v)] - [B(\hat x, u(\hat y), p^v) - B(\hat y, u(\hat y), p^v)] \nonumber \\
    =  &  F^{ij} (X^v_{ij} - Y^u_{ij}) - [F^{ij}D_{p_k}A_{ij}(\hat y, u(\hat y), p^u) + D_{p_k}B(\hat y, u(\hat y), p^u)] (p^v_k - p^u_k) \label{Theta} \\ 
        & - \frac{1}{2} [F^{ij}D_{p_kp_l}A_{ij}(\hat y, u(\hat y), \tilde p)+D_{p_kp_l}B(\hat y, u(\hat y), \bar p)]  (p^v_k - p^u_k)  (p^v_l - p^u_l)  \nonumber \\ 
        & -[F^{ij}D_xA_{ij}(\tilde x, u(\hat y), p^v) + D_xB(\bar x, u(\hat y), p^v)](\hat x - \hat y) \nonumber \\
   \le &  F^{ij} (X^v_{ij} - Y^u_{ij}) - [F^{ij}D_{p_k}A_{ij}(\hat y, u(\hat y), p^u) + D_{p_k}B(\hat y, u(\hat y), p^u)] (p^v_k - p^u_k) \nonumber \\
        & + C(1+\mathscr{T})[|p^v - p^u|^2 + |\hat x - \hat y|],  \nonumber     
\end{align}
where Taylor's formula and mean value theorem are used in the last equality with $\tilde p=tp^u+(1-t)p^v$, $\bar p=sp^u+(1-s)p^v$, $\tilde x=t^\prime \hat x+(1-t^\prime)\hat y$ and $\bar x=s^\prime \hat x+(1-s^\prime)\hat y$ for some $t, s, t^\prime, s^\prime \in (0,1)$, the constant $C$ depends on $\|A\|_{C^2(\bar \Omega \times \mathbb{R} \times \mathbb{R}^n)}$ and $\|B\|_{C^2(\bar \Omega \times \mathbb{R} \times \mathbb{R}^n)}$.
By calculations, we have
\begin{align}
D_x\phi(\hat x, \hat y) = & \frac{1}{\epsilon}a(z)(\hat x-\hat y) + \varphi(z,v(z))\beta(z) - \delta Dd(\hat x) + 2\delta (\hat x-z),  \label{phi x}\\
-D_y\phi(\hat x, \hat y) = & \frac{1}{\epsilon}a(z)(\hat x-\hat y) + \varphi(z,v(z))\beta(z) + \delta Dd(\hat y). \label{-phi y}
\end{align}
Observing the forms of $p^v, p^{\tilde u}$ in \eqref{semi-jets v} and \eqref{semi-jets tilde u}, together with \eqref{phi x} and \eqref{-phi y}, since $v(x)$ and $\tilde u(y)$ are Lipschitz continuous functions, we see that
\begin{equation}\label{using Lipschitz 1}
\frac{1}{\epsilon}a(z)(\hat x-\hat y)
\end{equation} 
is bounded independently of $\epsilon$, and consequently, $D_x\phi(\hat x, \hat y)$ in \eqref{phi x} and $-D_y\phi(\hat x, \hat y)$ in \eqref{-phi y} are bounded independently of $\epsilon$. Thus, there exists a positive constant $\kappa$ independent of $\epsilon$, such that
\begin{equation}\label{using Lipschitz 2}
|\hat x - \hat y| \le \kappa \epsilon.
\end{equation}
Using $p^v, p^u$ in \eqref{semi-jets v} and \eqref{pu}, we have
\begin{equation}\label{pvpu}
p^v - p^u = - \frac{\tau K \eta(\hat y)}{1+\tau K \eta(\hat y)} [D\underline u(\hat y) + D_y\phi(\hat x, \hat y)] - \delta [Dd(\hat x)+Dd(\hat y)-2(\hat x-z)]. \\
\end{equation}
Then, we have the estimate
\begin{equation}\label{pvpu2}
|p^v - p^u|^2  \le  C(\tau^2 + \delta), 
\end{equation}
for $\delta\in (0,1)$, where the constant $C$ depends on $K, \eta, \varphi, \beta, \Omega$ and $\|\underline u\|_{C^1(\bar \Omega)}$. Note that in the process of estimating \eqref{pvpu2}, the facts \eqref{using Lipschitz 1} and \eqref{using Lipschitz 2} have been taken into account. Plugging $X^v$, $Y^u$, and the inequlalities \eqref{using Lipschitz 2}, \eqref{pvpu}, \eqref{pvpu2} into \eqref{Theta}, and using $X\le Y$, we have
\begin{align}
\Theta(\hat x, \hat y) \le &  \tilde \Theta(\hat x, \hat y) - \delta F^{ij}[D_{ij}d(\hat x) + D_{ij}d(\hat y) - 2\delta_{ij}]  \label{Theta tilde Theta} \\
                                      & + \delta [F^{ij}D_{p_k}A_{ij}(\hat y, u(\hat y), p^u) + D_{p_k}B(\hat y, u(\hat y), p^u)] [D_kd(\hat x) + D_k d(\hat y) - 2(\hat x -z)_k]  \nonumber \\
                                      &  + C(1+\mathscr{T})[|p^v - p^u|^2 + |\hat x - \hat y|]  \nonumber \\
                                 \le &  \tilde \Theta(\hat x, \hat y) + C(\tau^2 + \delta + \epsilon) (1+ \mathscr{T}),  \nonumber    
\end{align}
for $\delta\in (0,1)$ and a further constant $C$, where
\begin{align}
\tilde \Theta(\hat x, \hat y) = & -\frac{\tau K \eta(\hat y)}{1+\tau K \eta(\hat y)} \big \{ F^{ij} [D_{ij}\underline u(\hat y) - Y_{ij} - \delta D_{ij}d(\hat y) \label{Theta 0} \\
                                    & + \frac{K}{[1+\tau K \eta(\hat y)]^2} (D_i\underline u(\hat y) + D_{y_i}\phi(\hat x, \hat y)) (D_j\underline u(\hat y) + D_{y_j}\phi(\hat x, \hat y)) ] \nonumber \\
                                    & - [F^{ij}D_{p_k}A_{ij}(\hat y, u(\hat y), p^u) + D_{p_k}B(\hat y, u(\hat y), p^u)]  [D_k\underline u(\hat y) + D_{y_k}\phi(\hat x, \hat y)] \big \}. \nonumber
\end{align}
We claim that
\begin{equation}\label{claim}
\tilde \Theta (\hat x, \hat y) \le -\tau \delta_2 (1+\mathscr{T}),
\end{equation}
for some positive constant $\delta_2$. From \eqref{Theta tilde Theta} and \eqref{claim}, by choosing $\tau \le \delta_2/(2C)$ and $\delta + \epsilon \le \tau \delta_2/ (4C)$, we have
\begin{equation}\label{Theta < 0}
\Theta (\hat x, \hat y) \le [-\tau \delta_2 + C(\tau^2 + \delta + \epsilon) ] (1+ \mathscr{T}) \le - \frac{\tau\delta_2}{4} (1+ \mathscr{T}) <0.
\end{equation}
Then \eqref{0 le Theta} and \eqref{Theta < 0} lead to a contradiction.

The remaining task is to prove the claim \eqref{claim}. First, it is readily checked that
\begin{align}
\tilde \Theta(\hat x, \hat y) = & - \tau K \eta(\hat y) \{F^{ij}[D_{ij}\underline u(\hat y) - Y_{ij}^u + K(D_i\underline u(\hat y) - p^u_i)(D_j\underline u(\hat y) - p^u_j)]   \label{Theta 1} \\
                                             & - [F^{ij}D_{p_k}A_{ij}(\hat y, u(\hat y), p^u) + D_{p_k}B(\hat y, u(\hat y), p^u)] (D_k\underline u(\hat y) - p^u_k)) \}.  \nonumber
\end{align}
Since $\underline u\in C^{2}(\bar \Omega)$ is a strictly subsolution of equation \eqref{1.1}, there exist positive constants $\bar \delta$ and $\bar \delta^\prime$ such that
\begin{equation}\label{strictly subsolution}
F(M[\underline u]-\bar \delta I) \ge B(\hat y, \underline u(\hat y), D\underline u(\hat y)) + \bar \delta^\prime.
\end{equation}
Using the monotonicity of $A$ and $B$ in $z$, $\underline u\ge u$ in $\Omega$, and the concavity condition F2, we have
\begin{align}
    & - \frac{1}{\tau K \eta(\hat y)} \tilde \Theta(\hat x, \hat y) \nonumber \\
\ge  & \bar \delta \mathscr{T} + F^{ij}\{ [D_{ij}\underline u(\hat y) - A_{ij}(\hat y, \underline u(\hat y), D\underline u(\hat y)) - \bar \delta \delta_{ij}]- [Y_{ij}^u-A_{ij}(\hat y, \underline u(\hat y), p^u)]  \} \nonumber \\
    & - B(\hat y, \underline u(\hat y), D\underline u(\hat y)) + B(\hat y, \underline u(\hat y), p^u) + KF^{ij}(D_i\underline u(\hat y) - p^u_i)(D_j\underline u(\hat y) - p^u_j)   \nonumber \\
    & + F^{ij}[ A_{ij}(\hat y, u(\hat y), D\underline u(\hat y)) -A_{ij}(\hat y, u(\hat y), p^u) - D_{p_k}A_{ij}(\hat y, u(\hat y), p^u) (D_k\underline u(\hat y) - p^u_k) ] \nonumber \\
    & + [B(\hat y, u(\hat y), D\underline u(\hat y)) - B(\hat y, u(\hat y), p^u) - D_{p_k}B(\hat y, u(\hat y), p^u) (D_k\underline u(\hat y) - p^u_k)]  \nonumber \\
\ge & \bar \delta \mathscr{T} + [F(M[\underline u]-\bar \delta I)-B(\hat y, \underline u(\hat y), D\underline u(\hat y))] - [F(Y^u - A(\hat y, u(\hat y), p^u))-B(\hat y, \underline u(\hat y), p^u) ] \label{estimate tilde Theta} \\
      & + \frac{1}{2} [F^{ij}D_{p_kp_l}A_{ij}(\hat y, u(\hat y), \tilde p)+D_{p_kp_l}B(\hat y, u(\hat y), \bar p)](D_k\underline u(\hat y) - p^u_k)(D_l\underline u(\hat y) - p^u_l) \nonumber \\
      & + KF^{ij}(D_i\underline u(\hat y) - p^u_i)(D_j\underline u(\hat y) - p^u_j) \nonumber  \\
\ge & (\bar \delta - \frac{\bar \lambda \epsilon_0}{2}|D\underline u(\hat y)-p^u|^2) \mathscr{T} + (K-\frac{\bar \lambda}{2\epsilon_0}) F^{ij} (D_i\underline u(\hat y) - p^u_i)(D_j\underline u(\hat y) - p^u_j) + \bar \delta^\prime,  \nonumber
\end{align}
for any positive constants $\epsilon_0\in (0, 1]$ and $K$, and a non-negative function $\bar \lambda \in C^0(\bar \Omega\times \mathbb{R} \times \mathbb{R}^n)$, where Taylor's formula is used in the second equality with $\tilde p=tD\underline u(\hat y)+(1-t)p^u$ and $\bar p=sD\underline u(\hat y)+(1-s)p^u$ for some $t, s\in (0,1)$, \eqref{strictly subsolution}, \eqref{u Yu inequ}, regularity of $A$, and convexity of $B$ in $p$ are used to obtain the last inequality. Note that here we use the inequality (2.3) in \cite{JT-oblique-II} for the regular condition of $A$, namely
\begin{equation}\label{equi regular}
F^{ij} (D_{p_kp_l}A_{ij} )\eta_k \eta_l \ge \bar \lambda (\epsilon_0 \mathscr{T}|\eta|^2 + \frac{1}{\epsilon_0}F^{ij}\eta_i \eta_j),
\end{equation}
for any non-negative symmetric matrix $\{F^{ij}\}$, $\eta \in \mathbb{R}^n$ and $\epsilon_0\in (0, 1]$, where $\bar \lambda \in C^0(\bar \Omega\times \mathbb{R} \times \mathbb{R}^n)$ is a non-negative function.
By successively fixing $\epsilon_0 \le \bar \delta / {\sup_\Omega (\bar \lambda |D\underline u -p^u|)^2}$ and $K\ge (\sup_\Omega \bar \lambda)/ (2\epsilon_0)$ in \eqref{estimate tilde Theta}, we have
\begin{equation}
- \frac{1}{\tau K \eta(\hat y)} \tilde \Theta(\hat x, \hat y) \ge \frac{\bar \delta_0}{2}  \mathscr{T}  +  \bar \delta^\prime,
\end{equation}
which leads to the claim \eqref{claim} by taking $\delta_2 = \min \{\frac{\bar \delta_0}{2}, \bar \delta^\prime\} K\inf_\Omega \eta$. We have finished the verification of claim \eqref{claim} and completed the proof of Theorem \ref{Th4.3} in case (i).

Next, we consider the cases (ii) and (iii).
When either $A$ or $B$ is strictly increasing in $z$, we observe that for $u, v \in C^{1,1}(\Omega)\cap C^{0,1}(\bar \Omega)$, $v-u$ only attains its positive maximum at a point $z\in \partial\Omega$, which is implied by using Bony maximum principle and similar argument in \eqref{max F Th3.2}. Then we consider the function $\Phi(x,y)$ in \eqref{Phi x y} with $\phi(x,y)$ defined in \eqref{phi x y}, (which corresponds to $\tilde \Phi(x,y)$ in \eqref{tilde Phi x y} with $\tau=0$).
Assuming that $\Phi(x,y)$ takes its maximum at $(\hat x, \hat y)$,
by Ishii's Lemma, there exists $X, Y\in \mathbb{S}^n$ such that \eqref{semi-jets v}, \eqref{semi-jets tilde u} with $\tilde u$ replaced by $u$, and \eqref{a a a a} hold. 
Note that \eqref{a a a a} implies $X\le Y$. Thus, if $\hat x\in \partial\Omega$, by the definition of viscosity subsolution $v$ under F1$^-$, we have the ineqaulity \eqref{v boundary inequality} or \eqref{v X inequ}.
From \eqref{G xi x 1}, the only possible case is \eqref{v X inequ}. Note that if $\hat x\in \Omega$, the inequality \eqref{v X inequ} holds directly from the definition of the viscosity subsolution $v$.
Similarly, if $\hat y\in \partial\Omega$, by the definition of viscosity subsolution $u$ under F1$^-$, then either \eqref{u boundary inequality pu} with $p^u=- D_y\phi(\hat x, \hat y)$ or \eqref{u Yu inequ} with $p^u=- D_y\phi(\hat x, \hat y)$ and $Y^u=Y+\delta D^2d(\hat y)$ holds.
In view of \eqref{G xi x 2}, the only possible case is the latter case. Note that if $\hat y\in \Omega$, the inequality \eqref{u Yu inequ} holds directly from the definition of the viscosity supersolution $u$. 

Using the inequalities \eqref{v X inequ} and \eqref{u Yu inequ}, and following similar calculations from \eqref{v>u} to \eqref{RB}, together with \eqref{strictly increasing A} and \eqref{Theta function 2} in case (ii), \eqref{strictly increasing B}, \eqref{bounded trace} and \eqref{Theta function 3} in case (iii), with $Dv(\hat x), Du(\hat y), D^2v(\hat x)$ and $D^2u(\hat y)$ replaced by $p^v=D_x\phi(\hat x, \hat y), p^u=- D_y\phi(\hat x, \hat y), X^v=X-\delta D^2d(\hat x) + 2\delta I$ and $Y^u=Y+\delta D^2d(\hat y)$ respectively,  we then get contradictions and complete the proof of Theorem \ref{Th4.3} in cases (ii) and (iii).
\end{proof}

\begin{remark}
In fact, the inequality in claim \eqref{claim} is just the barrier inequality \eqref{barrier inequality} in the viscosity sense. 
When $u$ is $C^2$ at $\hat y$, we have in \eqref{Theta 1},
\begin{equation}
\tilde \Theta(\hat x, \hat y) = - \tau \mathcal{L} \tilde \eta(\hat y),
\end{equation}
where the operator $\mathcal{L}$ is the same as that in the barrier inequality \eqref{barrier inequality} in Lemma \ref{Lemma 4.1}.
\end{remark}

\begin{remark}
Note that the key technique of doubling variables and penalization used in the proof of Theorems \ref{Th4.3} is the Ishii's Lemma in \cite{K2004}, which has its origins in \cite{Ishii1989, IL1990}, (see also Theorem 3.2 in \cite{CIL1992}).
\end{remark}

\begin{remark}
From the existence theorems, the solution $u$ of the oblique boundary value problem \eqref{1.1}-\eqref{1.5} will be natural in the class $C^2(\Omega)\cap C^{0,1}(\bar \Omega)$ ($C^{2,\alpha}(\Omega)\cap C^{0,1}(\bar \Omega)$) under F1, and in the class $C^{1,1}(\Omega)\cap C^{0,1}(\bar \Omega)$ under F1$^-$, respectively. Therefore, in Theorems \ref{Th4.2} and \ref{Th4.3}, we study the comparison principles for solutions in  $C^2(\Omega)\cap C^{0,1}(\bar \Omega)$ and $C^{1,1}(\Omega)\cap C^{0,1}(\bar \Omega)$ respectively. In fact, by proper adjustment of the proof, the conclusion in Theorem \ref{Th4.3} can also hold for $u, v\in C^{0,1}(\bar \Omega)$ in case (i) and $u, v \in C^{0}(\bar \Omega)$ in cases (ii) and (iii), which will be taken up in a sequel. Moreover when $A$ is strictly regular, $C^{0,1}(\Omega)$ viscosity solutions are $C^{1,1}$ smooth in the degenerate case and $C^{2,\alpha}$ smooth in the non-degenerate case. 
\end{remark}

\begin{remark}
The comparison principles in Theorems \ref{Th4.2} and \ref{Th4.3} are proved in cases (i), (ii) and (iii) respectively. 
The assumptions in cases (ii) and (iii) correspond to the assumptions in Theorem VI.3 in \cite{IL1990}, see also \cite{Tru1988, Ishii1991, CIL1992, B1993}.
The study of the comparison principle in case (i) when $\varphi$ is strictly increasing in $z$ is missing in the literature; here we prove it by making full use of the barrier in \cite{JT-oblique-II} with the help of the strict subsolution.
\end{remark}


With these comparison principles, we are ready to complete the proof of Theorem \ref{Th1.1}.

\begin{proof}[Proof of Theorem \ref{Th1.1}]
The existence in assertion (i) has already been proved in Section \ref{Section 3}.
If $A$ is strictly increasing in $z$, the uniqueness for solutions in $C^{1,1}(\Omega)\cap C^{0,1}(\bar \Omega)$ of the boundary value problem \eqref{1.1}-\eqref{1.5} follows from case (ii) of Theorem \ref{Th4.3}. While if $A$ and $B$ are independent of $z$, and $\varphi$ is strictly increasing in $z$, taking account of Remark \ref{Rem about subsol}, the uniqueness for solutions in $C^{1,1}(\Omega)\cap C^{0,1}(\bar \Omega)$ of the boundary value problem \eqref{1.1}-\eqref{1.5} follows from case (i) of Theorem \ref{Th4.3}. We thus complete the proof of the uniqueness result in assertion (ii) of Theorem \ref{Th1.1}.

From the uniqueness, the local regularity result in assertion (iii) is immediate. In fact, under F1$^-$, when $\Gamma \subset \mathcal P_{n-1}$ and $\partial\Omega$ is uniformly $(\Gamma, A, G)$-convex at $x_0\in \partial\Omega$ with respect to $u$, the assumptions for the local second derivative estimate in Theorem \ref{Th2.1} are also satisfied for $u_\epsilon$ of the regularized problem \eqref{re-problem}. By \eqref{est int bdy} in Theorem \ref{Th2.1}, we have
\begin{equation}\label{uni N bd}
\sup_{\mathcal{N}\cap \Omega} |D^2 u_\epsilon| \le C,
\end{equation}
where $\mathcal{N}=B_{\theta R}(x_0) \cap \Omega$ with $x_0\in \partial\Omega$ and $\theta\in (0,1)$, $C$ is a constant depending on $\theta, B_R\cap \Omega, F, A, B, \varphi, \beta, \beta_0$ and $|u_\epsilon|_{1;B_R\cap \Omega}$, and is independent of $\epsilon$. Recalling the local gradient estimate in Theorem 3.1 in \cite{JT-oblique-I}, the quantity $|u_\epsilon|_{1;B_R\cap \Omega}$ depends on $B_R\cap \Omega, F, A, B, \varphi, \beta$ and $|u_\epsilon|_{0;B_R\cap \Omega}$, and is independent of $\epsilon$. The locally uniform bound for $|u_\epsilon|_{0;B_R\cap \Omega}$ is obvious from \eqref{uni sol}. With the uniform bound in \eqref{uni N bd}, by repeating the limiting process $\epsilon\rightarrow 0$ in the F1$^-$ case of the proof for Theorem \ref{Th3.1} and considering the uniqueness, we obtain
\begin{equation}
u\in C^{1,1}(\Omega\cup (\mathcal N\cap \bar\Omega)),
\end{equation}
which completes the proof of assertion (iii).

From Theorem \ref{Th4.2}, we have the uniqueness for solutions in $C^{2,\alpha}(\Omega)\cap C^{0,1}(\bar\Omega)$ ($\alpha\in (0,1)$) of the problem \eqref{1.1}-\eqref{1.5}. Then the proof of assertion (iv) is parallel to that of assertion (iii), which can be done by using \eqref{uni N bd} and repeating the limiting process $\epsilon\rightarrow 0$ in the F1 case of the proof for Theorem \ref{Th3.1}. To avoid the repetitions, here we omit its detailed proof.
\end{proof}

With the alternative hypotheses of gradient estimate in Remark \ref{Rem 1.1} and Section \ref{Section 3}, we state the resultant application to Theorem \ref{Th1.1} as a corollary.
\begin{Corollary}\label{Cor 4.1}
Theorem \ref{Th1.1} continues to hold when the condition that A is uniformly regular is replaced by A strictly regular, $|\frac{\beta}{\beta\cdot\nu} - \nu|<1/\sqrt{n}$ and $\mathcal{F}$ is orthogonally invariant, together with any of the conditions (a), (b) or (c) in Remark \ref{Rem 1.1}.
\end{Corollary}

When $B>a_0=-\infty$ in the F1 case of Theorem \ref{Th1.1} and Corollary \ref{Cor 4.1}, the assertions still hold if $A$ and $B$ satisfy the growth conditions \eqref{quadratic growth} with $B\ge O(1)$ and \eqref{3.26} holds in cases (b) and (c). Furthermore we may also replace the cone $\Gamma$ in Theorem \ref{Th1.1} and Corollary \ref{Cor 4.1} by an open convex set $D$ as indicated in Remark \ref{Rem 1.4}. Note here that we can also embrace the case $a_0 = -\infty$ in Corollary \ref{Cor 4.1}, by taking $D = \{r\in\Gamma | F > a_0\}$, for some finite $a_0 \le B$, as in \cite{ITW2004}.


\end{document}